\documentclass[graybox]{svmult}

\usepackage{mathptmx}       % selects Times Roman as basic font
\usepackage{helvet}         % selects Helvetica as sans-serif font
\usepackage{courier}        % selects Courier as typewriter font
\usepackage{type1cm}        % activate if the above 3 fonts are
\usepackage{makeidx}         % allows index generation
\usepackage{graphicx}        % standard LaTeX graphics tool
\usepackage{multicol}        % used for the two-column index
\usepackage[bottom]{footmisc}% places footnotes at page bottom

\usepackage{amsmath,float}
\usepackage{amssymb}
\usepackage{amscd}
\usepackage[OT2,T1]{fontenc}
\usepackage{indentfirst}

\def\N{{\mathbb N}}
\def\Z{{\mathbb Z}}
\def\re{\mathrm{Re}}
\def\C{{\mathbb C}}
\def\Q{{\mathbb Q}}
\def\det{\mathrm{det}}

\DeclareSymbolFont{cyrletters}{OT2}{wncyr}{m}{n}\DeclareMathSymbol{\Sha}{\mathalpha}{cyrletters}{"58}

\def\Tr{\mathop{\rm Tr}}

\def\re{\mathrm{Re}}
\def\ord{\mathop{\rm ord}}

\def\Gal{\mathrm{Gal}}

\def\Frob{\mathrm{Frob}}
\newcommand{\Tam}{\operatorname{Tam}}
\newcommand{\Directsum}{\bigoplus}

\def\GL{\mathrm{GL}}

\def\ord{{\textrm{ord}}}
\def\1{{\bf 1}}

\makeindex             % used for the subject index
\begin{document}

\title*{Explorations in the theory of partition zeta functions
\thanks{The first author is supported by National Science Foundation and the Asa Griggs Candler Fund.}}
\titlerunning{Explorations in the theory of partition zeta functions}
\author{Ken Ono, Larry Rolen, and Robert Schneider}
\authorrunning{K. Ono, L. Rolen, and R. Schneider}
\institute{Ken Ono \at Department of Mathematics and Computer Science, Emory University \\ \email{ono@mathcs.emory.edu} \\ Larry Rolen \at Department of Mathematics, The Pennsylvania State University \\ \email{larryrolen@psu.edu}\\ Robert Schneider \at Department of Mathematics and Computer Science, Emory University\\ \email{robert.schneider@emory.edu}}

\maketitle

\abstract{}
%Here we survey zeta functions arising from the theory of integer partitions. 
%We explore %and explore 
%the resulting theory using classical instruments such as Bell polynomials and Fa\`{a} di Bruno's formula, as well as ideas related to the theory of multiplicative partitions. % studied by Chamberland, Johnson, Nadeau, and Wu and by Chamberland and Straub
We introduce and survey results on two families of zeta functions connected to the multiplicative and additive theories of integer partitions.  In the case of the multiplicative theory, we provide specialization formulas and results on the analytic continuations of these ``partition zeta functions'', find unusual formulas for the Riemann zeta function, prove identities for multiple zeta values, and see that some of the formulas allow for $p$-adic interpolation. The second family we study was anticipated by Manin and makes use of modular forms, functions which are intimately related to integer partitions by universal polynomial recurrence relations. We survey recent work on these zeta polynomials, including the proof of their Riemann Hypothesis.

%%%%%%%%%%%%%%%%%%%%%%%%%%%%%%%%%%%%%%%%%%%%%%%%%%%%%%%%%%%%%%%%%%%%%%%%%%%%%%%%%%%%%%%%%%%
%%%%%%%%%%%%%%%%%%%%%%%%%%%%%%%%%%%%%%%%%%%%%%%%%%%%%%%%%%%%%%%%%%%%%%%%%%%%%%%%%%%%%%%%%%%

\section{The setting: Visions of Euler}

In antiquity, storytellers began their narratives by invoking the muse whose divine influence would guide the unfolding imagery. It is fitting, then, that we begin this article by praising the immense curiosity of Euler, whose imagination ranged playfully across almost the entire landscape of modern mathematical thought. %What ecstasy he must have experienced, to be the first to hear the elegant music of combinatorics and complex analysis, to see the convergence of exotic infinite series like fractal trails in a hall of mirrors, to find a wormhole connecting the alien realms of infinite products and sums \cite{Dunham}.%, to feel invisible spaces ripple and respond under his every poke and prod. 

%Only three generations before Euler's time, Descartes had discovered the bridge between Algebra and Geometry. He traced the evanescent landscape of thought in terse heiroglyphs, just the right format for mortal minds: little scribbles pregnant with enormousness. Barrow, Leibniz, Newton, Wallis, and their age cultivated the landscape of Descartes; the Bernoulli family harvested these fruits and began a cottage industry of The Calculus. 

%And so it was that Johann Bernoulli took on a  young pupil, raised closer to the soil of Analysis than any child since the dawn of time. This was Euler, and it was not mere fate that the child---bright, joyful, pure of heart---should grow up to encompass the infinite in his daydreams. It was divine intervention in human history: Mathematics reached out to touch our tiny hands.

Euler made spectacular use of product-sum relations, often arrived at by unexpected avenues, thereby inventing one of the principle archetypes of %a large corner of 
modern number theory. Among his many profound identities is the product formula for what is now called the Riemann zeta function:
\begin{equation}\label{zetaproduct}
\zeta(s):=\sum_{n=1}^{\infty}n^{-s}=\prod_{p\in\mathbb P}(1-p^{-s})^{-1}
\end{equation}

With this relation, Euler connected the (at the time) cutting-edge theory of infinite series to the ancient set $\mathbb P$ of prime numbers. Moreover, in solving the famous ``Basel problem'' posed a century earlier by Pietro Mengoli (1644), Euler showed us how to compute even powers of $\pi$ using the zeta function, giving explicit formulas of the shape 
\begin{equation}\label{zeta_even}
\zeta(2N)=\pi^{2N}\times\  \mathrm{rational}.
\end{equation}
It turns out there are other classes of zeta functions, arising from other Eulerian formulas in the universe of partition theory. %, which (among many connections they make) allow us to calculate odd powers of $\pi$ as well. {\bf RPS: We need at least one more example of this odd powers thing, or delete comment, but it is an interesting angle Ken pointed out.} 

Much like the set of positive integers, but perhaps even more richly, the set of integer partitions ripples with striking patterns and beautiful number-theoretic phenomena. In fact, the positive integers $\mathbb N$ are embedded in the integer partitions $\mathcal P$ in a number of ways: obviously, positive integers themselves represent the set of partitions into one part; less trivially, the prime decompositions of integers are in bijective correspondence with the set of prime partitions, i.e., the partitions into prime parts (if we map the number 1 to the empty partition $\emptyset$), as Alladi and Erd\H{o}s note \cite{AlladiErdos}. We might also identify the divisors of $n$ with the partitions of $n$ into identical parts, and there are many other interesting ways to associate integers to the set of partitions.

Euler found another profound product-sum identity, the generating function for the partition function $p(n)$ %{\bf RPS: removed colon, I think is not correct format for this sentence structure}
\begin{equation}\label{partition_genfctn}
\prod_{n=1}^{\infty}(1-q^n)^{-1}=\sum_{n=0}^{\infty}p(n)q^n,
\end{equation}
single-handedly establishing the theory of integer partitions. This formula doesn't look much like the zeta function identity \eqref{zetaproduct}; however, generalizing Euler's proofs of these theorems leads to a new class of partition-theoretic zeta functions. It turns out that \eqref{zetaproduct} and \eqref{partition_genfctn} both arise as specializations of a single ``master'' product-sum formula.

Before we proceed, let us fix some notation. Let $\mathcal P$ %(resp. $\mathcal P^*$) 
denote the set of all integer partitions%(resp. partitions into distinct parts)
. %Let $\mathcal P_X$ (resp. $\mathcal P_X^*$) denote the set of partitions (resp. partitions into distinct parts) into elements of $X\subseteq \mathbb N$. 
Let $\lambda=(\lambda_1,\lambda_2,\dots,\lambda_r),$ with $\lambda_1\geq\lambda_2\geq\dots\geq\lambda_r\geq 1$, denote a generic partition, $l(\lambda):=r$ denote its {\it length} (the number of parts), and $|\lambda|:=\lambda_1+\lambda_2+\dots+\lambda_r$ denote its {\it size} (the number being partitioned). We write ``$\lambda\vdash n$'' to mean $\lambda$ is a partition of $n$, and ``$\lambda_i\in\lambda$'' to indicate $\lambda_i\in\mathbb N$ is one of the parts of $\lambda$.

Then we have the following ``master'' identity (which is a piece of Theorem 1.1 in \cite{Robert}). % {\bf LR: Can you say it comes from your paper and which theorem there so it is clear that we aren't claiming this as new in this paper?} %{\bf RPS: Larry, we do not cover partition zeta functions into distinct parts at all, except for if we note that pertinent MZVs are of this form... should I omit the second equation in the following proposition and we just leave out distinct parts all together from the treatise, or maybe just state somewhere that the case of distinct parts is also covered if it comes up in a proof? Or should we leave it in, as it is useful and for completeness? LR: It is in your last paper, right? I think we can just leave the first formula in, as it isn't as related to the main results here, and we are just giving a small taste of more general related things. RPS: Done.} %, which naturally generalize Euler's partition generating sum formula \ref{partition_genfctn}.
 
\begin{proposition}\label{productsum} For an arbitrary function $f:\N\to\C$, we have %the pair of identities
$$
\prod_{n=1}^{\infty}(1-f(n)q^n)^{-1}=\sum_{\lambda\in\mathcal P}q^{|\lambda|}\prod_{\lambda_i\in\lambda}f(\lambda_i).
$$
%$$  
%\prod_{n=1}^{\infty}(1-f(n)q^n)=\sum_{\lambda\in\mathcal P^*}(-1)^{\ell(\lambda)}q^{|\lambda|}\prod_{\lambda_i\in\lambda}f(\lambda_i)
\end{proposition}
 
The sum on the right is taken over all partitions, and the left-hand product is taken over the parts $\lambda_i$ of partition $\lambda$. The proof proceeds along similar lines to Euler's proof of \eqref{partition_genfctn}, and can be seen immediately if one expands a few terms of the infinite product by hand, without collecting coefficients in the usual manner. This simple identity also has interesting (and sometimes exotic) representations in terms of Eulerian $q$-series and continued fractions; readers are referred to \cite{Robert} for details. 

\begin{remark}
Note that, collecting like terms, we can re-write such partition sums as standard power series, summed over non-negative integers:
\begin{equation}\label{rewrite}
\sum_{\lambda\in\mathcal P}c_{\lambda}q^{|\lambda|}=\sum_{n=0}^{\infty}q^n\sum_{\lambda\vdash n}c_{\lambda}
\end{equation}
\end{remark}

It is immediate from Proposition \ref{productsum} and \eqref{rewrite} that the partition generating function (\ref{partition_genfctn}) results if we let $f\equiv 1$ identically. 

Similarly, if we set $f(n)=n^{-s},q=1$, and sum instead over the subset $\mathcal P_{\mathbb P}$ of prime partitions, we arrive at the Euler product formula (\ref{zetaproduct}) for the zeta function. This follows from the bijective correspondence between prime partitions and the factorizations of natural numbers noted above.

In light of these observations, we might view the Riemann zeta function as the prototype for a new class of combinatorial objects arising from Eulerian methods.

\subsection{Partition-theoretic zeta functions}\label{Partition-theoretic zeta functions}

Inspired by work of Euler \cite{Dunham}, Fine \cite{Fine}, Andrews \cite{Andrews}, Bloch and Okounkov \cite{BlochOkounkov}, Zagier \cite{Zagier2}, Alladi and Erd\H{o}s \cite{AlladiErdos}, and others, the authors here undertake the study of a class of zeta functions introduced by the third author in \cite{Robert}, resembling the Riemann zeta function $\zeta(s)$ but summed over proper subsets of $\mathcal P$, as opposed to over natural numbers. 

In this paper, we review a few of the results from \cite{Robert}, and record a number of further identities relating certain zeta functions arising from the theory of partitions to various objects in number theory such as Riemann zeta values, multiple zeta values, and infinite product formulas. Some of these formulas are related to results in the literature; they are presented here as examples of this new class of partition-theoretic zeta functions. We also give several formulas for the Riemann zeta function, and results on the analytic continuation (or non-existence thereof) of zeta-type series formed in this way. Furthermore, we discuss the $p$-adic interpolation of these zeta functions in analogy with the classical work of Kubota and Leopoldt on $p$-adic continuation of the Riemann zeta function \cite{KL}.

%The proofs are collected at the end of each section{\bf RPS: to keep the prose and structure uncluttered and elegant}, and are suppressed where they might be easily provided by the reader. 

To describe our primary object of study, we must introduce a new statistic related to partitions, to complement the length and size defined above. We define the {\it integer} of a partition $\lambda$, notated as $n_{\lambda}$, to be the product of its parts:
\begin{equation}
n_{\lambda}:=\lambda_1 \lambda_2 ... \lambda_r
\end{equation}
This multiplicative statistic may not seem very natural as partitions arise purely additively, with deep additive structures such as Ramanujan congruences dominating the theory. Yet if we let $n_{\lambda}$ formally replace the usual index $n$ in the Riemann zeta function, and sum instead over appropriate subsets of partitions, we arrive at an analytic-combinatorial object with many nice properties. 

\begin{definition}\label{zetadef}

Over a subset $\mathcal P'\subset\mathcal P$ and value $s\in\C$ for which the sum converges%{\bf LR: This seems funny, as we say where it converges, and we say below a region of convergence we want. Is it clear that this particular right half-plane is always the right one? RPS: Is this preceding sentence better now, with $s$ removed from the equation below?}
, we define a \emph{partition zeta function} to be the series

\begin{equation*}
\zeta_{\mathcal P'}(s):=\sum_{\lambda\in\mathcal P'}n_{\lambda}^{-s}.
\end{equation*}

\end{definition}

%REMOVED FROM EQUATION ABOVE:  \  \left(\operatorname{Re}(s)>1\right)

Very nice relations %, as well as exotic formulas, 
arise from unique properties of special subsets $\mathcal P'$, e.g. partitions with some distinguishing structure, or into parts sharing some trait. For example, if we let $\mathcal P^*\subset\mathcal P$ denote partitions into distinct parts, there are interesting closed-form expressions, such as 
\begin{equation}
\zeta_{\mathcal P^*}(2)=\frac{\operatorname{sinh \pi}}{\pi},
\end{equation}
which follows from Proposition \ref{productsum} together with a formula of Euler (cf. \cite{Robert}). By summing instead over partitions into parts $\geq 2$ (that is, disallowing $1$'s), we arrive at curious identities such as
\begin{equation}\label{ramanujan}
\zeta_{\mathcal P_{\geq 2}}(3) = \frac{3\pi }{\cosh \left(\frac{1}{2} \pi \sqrt{3}\right)},\
%\zeta_{\mathcal P_{\geq 2}}(6) &= \frac{6\pi^2 }{\cosh^2\left(\frac{1}{2} \pi \sqrt{3}\right)},
\end{equation}
which can be found from Proposition \ref{productsum} together with a formula of Ramanujan \cite{Robert}. This equation is somewhat surprising, as the Riemann zeta values at odd arguments are famously enigmatic.  % (see \cite{vanderPoorten} for more).

Other attractive closed sums can be found---%{\bf LR: Do these dashes need to be this long? RPS: I think so but we could shorten to -- throughout if you think looks better. To my eye this is the correct dash for this context but it is subjective typographical style I think.}
and general structures observed, as detailed in \cite{Robert}---when we restrict our attention to partitions $\mathcal P_{\mathcal M}$ whose parts all lie in a subset $\mathcal M\subset\mathbb N$. %, and also if we sum over partitions $\mathcal P^*$ (resp. $\mathcal P_{\mathcal M}^*$) into distinct parts (resp. distinct parts lying in $\mathcal M$). 
It is easy to check from Proposition \ref{productsum} that we have the Euler product formula
\begin{align}\label{DirichletProduct}
\zeta_{\mathcal P_{\mathcal M}}(s)&=\prod_{k\in\mathcal M}\left(1-k^{-s}\right)^{-1}.%\\  
%\zeta_{\mathcal P_{\mathcal M}^*}(s)&=\prod_{k\in\mathcal M}\left(1+k^{-s}\right)
\end{align}
%\begin{remark}

We see from the right-hand side of \eqref{DirichletProduct} that $\zeta_{\mathcal P_{\mathcal M}}(s)$ diverges if $1\in\mathcal M$, thus the restriction $1\notin\mathcal M$ exhibited in the pair of identities above is a necessary one here. In fact, some restriction on the maximum multiplicity of 1 occurring as a part is necessary for any partition zeta function to converge. That is, we must sum over partitions containing 1 with multiplicity at most $m\geq 0$. %, as otherwise the contribution $n_{\lambda}^{-s}$ of any ``integer'' is repeated infinitely often as a summand (adjoining 1's to a partition does not affect the integer). 
Note that the resulting zeta function will equal the one in the case where 1's are not allowed, multiplied by $m+1$, as each $n_{\lambda}^{-s}$ is repeated $m+1$ times in the sum (adjoining 1's to a partition does not affect its ``integer''). %incarnations.
%\end{remark}

We note that if Dirichlet series coefficients $a_n$ are defined by
\[
\zeta_{\mathcal P_{\mathcal M}}(s)=:\sum_{n\geq1}a_n n^{-s}
,
\]
it is easy to see that $a_n$ counts the number of ways to write $n$ as a product of integers in $\mathcal M$, where each ordering of factors is only counted once. When $\mathcal M=\N$%{\bf RPS: Note: need to uniformize $\N$ vs. $\mathbb Z$ throughout LR: Made all $\N$}
, then these ways of writing $n$ as a product of smaller numbers are known as {\it multiplicative partitions}, and have been studied in a number of places in the literature; for example, the interested reader is referred to \cite{Andrews,ChamberlandJohnsonNadeauWu,MacMahon,Subbarao,ZaharescuZaki}.

We wish to study partition zeta functions over special subsets of $\mathcal P$ and arguments $s$ that lead to interesting relations. We begin by highlighting a few nice-looking examples. Let us recall a few identities from \cite{Robert} as examples of zeta function phenomena induced by suitable partition subsets of the form $\mathcal P_{\mathcal M}$. By summing over partitions into even parts we have a combinatorial formula to compute $\pi$:
\begin{equation}\label{pi/2}
\zeta_{\mathcal P_{2\mathbb N}}(2)=\frac{\pi}{2}
\end{equation}

Curiously, this partition zeta function produces an {\it odd} power of $\pi$, which does not occur in Euler's evaluations of $\zeta$ at even arguments. This simple formula belongs to an infinite family of %not-so-simple 
identities arrived at by letting $z=\pi/m$ in Euler's product identity for the sine function
\begin{equation}\label{sine}
\operatorname{sin} z = z\prod_{n=1}^{\infty} \left(1 - \frac{z^2}{\pi^2n^2}\right).
\end{equation}
To describe this family, we take $\mathcal P_{m\mathbb N} \subsetneq \mathcal P$ to denote the set of partitions into multiples of $m>1$%{\bf RPS: Need to uniformize this notation for subset LR: I didn't find any other uses that were inconsistent}, which are generated by taking $z=\pi /m$ in \eqref{sine}{\bf RPS: Need to uniformize equation listings, add parentheses and whatnot, throughout LR: A few notes: it is easier to just use eqref, instead of ref, and I had changed a few places before, since we shouldn't call things equation (number), but just (number).}
. Summing over these partitions, we have identities such as
\begin{equation}\label{X}
\zeta_{\mathcal P_{m\mathbb N}}(2) = \frac{\pi}{m \, \sin \left(\frac{\pi}{m}\right)},\end{equation}

\begin{equation}\label{XX}
\zeta_{\mathcal P_{m\mathbb N}}(4) = \frac{\pi^2}{m^2 \, \sin \left(\frac{\pi}{m}\right) \sinh \left(\frac{\pi}{m}\right)},
\end{equation}
and similar (but increasingly complicated) formulas can be computed for $\zeta_{\mathcal P_{m\mathbb N}}(2^t)$ with $t \in \mathbb N$. Clearly \eqref{pi/2} is the case $m=2$ of the first identity.

All of these partition zeta formulas seem to hint at analogous algebraicity results like that of \eqref{zeta_even} for the Riemann zeta function. Although the first few examples show that the situation is more complicated, such formulas immediately lead one to suspect the existence of a similar theory. %We will return to this question and offer such a general theory. Firstly, however, we also highlight a sample evaluation formula for a different sort of partition zeta function. %{\bf LR: It comes across a little as saying they aren't as elegant as the classical theory. I reworded it to make it sound stronger. RPS: Agreed, good job.}

%{\bf LR: robert, do we study the generalization of this one later? RPS: No, I don't guess so, but it is catchy and a little surprising as we let $s=3$, and shows it doesn't all follow trivially from ideas of Euler... but we can omit if you think we should.}

There is indeed a class of partition zeta functions with a structure much like that exhibited by classical zeta functions.% Let us restrict our attention to partitions of fixed length $k$.% (which were, incidentally, historically among the first subsets of $\mathcal P$ to be examined).

\begin{definition}\label{11}
We define the partition zeta function $\zeta_{\mathcal P}(\{s\}^k)$ to be the following sum over all partitions $\lambda$ of fixed length $\ell(\lambda) = k\geq 0$: %{\bf LR: Should we also say something about where it converges?}
\begin{equation*}%\label{11}
\zeta_{\mathcal P}(\{s\}^k) := \sum_{\ell(\lambda) = k} \frac{1}{n_{\lambda}^s}\  \  \  \left(\operatorname{Re}(s)>1\right) 
\end{equation*}
\end{definition}

We directly find that $\zeta_{\mathcal P}(\{s\}^0) = n_{\emptyset}^{-s} = 1$ and  $\zeta_{\mathcal P}(\{s\}^1) = \zeta(s)$. Proceeding much as Euler did to evaluate $\zeta(2k)$ (cf. \cite{Dunham}), for $k\geq 0$ we find that $\zeta_{\mathcal P}(\{2\}^k)$ is a rational multiple of $\zeta(2k)$:

\begin{equation}\label{pzv}
\zeta_{\mathcal P}(\{2\}^k)  = \frac{2^{2k - 1} - 1}{2^{2k-2}}\zeta(2k)
\end{equation}

%For example, we have the following values:
%\begin{align*}
%\zeta_{\mathcal P}(\{2\}) &= \zeta(2) = \frac{\pi^2}{6}, \\
%\zeta_{\mathcal P}(\{2\}^2) &= \frac{7}{4}\zeta(4) = \frac{7\pi^4}{360}, \\
%\zeta_{\mathcal P}(\{2\}^3) &= \frac{31}{16}\zeta(6)= \frac{31\pi^6}{15120},\dots , \\
%\zeta_{\mathcal P}(\{2\}^{13}) &= \frac{33554431}{16777216}\zeta(26) = \frac{22076500342261\pi^{26}}{93067260259985915904000000},\dots
%\end{align*}
%

It follows that $\zeta_{\mathcal P}(\{2\}^k)$ is of the form $``\pi^{2k} \times \text{rational}"$ for all positive $k$, much like the zeta values $\zeta(2k)$ given by Euler. Equation \eqref{pzv} also suggests the well-known Riemann zeta value 
$$\zeta(0) = \frac{2^{-2}}{2^{-1} - 1}\zeta_{\mathcal P}(\{2\}^0) = -1/2,$$
which is usually arrived at via analytic continuation. Furthermore, we can write $\zeta_{\mathcal P}(\{2^t\}^k)$ explicitly for all $t \in \mathbb N$ as finite combinations of well-known zeta values, also of the shape $``\pi^{2^tk} \times \text{rational}"$ \cite{Robert}.

Of course, it was Riemann's creative analytic continuation of the zeta function that led to stunning advances in number theory and analysis. Using general structural relations for partition zeta functions, the analytic continuation of $\zeta_{\mathcal P}(\{s\}^k)$%,\  \operatorname{Re}(s) > 1,$ {\bf LR: I don't think this range is needed here}
 is given in \cite{Robert} for one case: for fixed length $k=2$, we can write
\begin{equation}\label{shuffle}
\zeta_{\mathcal P}(\{s\}^2) = \frac{\zeta(2s) + \zeta(s)^2}{2}.
\end{equation}
Thus $\zeta_{\mathcal P}(\{s\}^2)$ inherits analytic continuation from the two zeta functions on the right. One would very much like to see further examples of analytic continuations of $\zeta_{\mathcal P}(\{s\}^k)$ %---perhaps even of $\zeta_{\mathcal P}(\{s\}^x)$ for $x\in\mathbb R$, whatever---
as well as other partition zeta functions.

We remark that \eqref{shuffle} resembles the well-known ``series shuffle product'' relation for multiple zeta values (see \cite{Besser}). Hoffman gives lovely formulas using different notations, relating $\zeta_{\mathcal P}(\{s\}^k)$ to combinations of multiple zeta values \cite{Hoffman1}. As we will see below, the theory of partition zeta functions connects with multiple zeta values in other ways as well. 

\subsection{Evaluations}
In the previous section we saw a variety of simple closed forms for partition zeta functions, depending on the natures of the subsets of partitions being summed over. Different subsets induce different zeta phenomena. %; we are interested to evaluate $\zeta_{\mathcal P'}(s)$ on different subsets $\mathcal P'\subsetneq\mathcal P$. 
In what follows, we consider the evaluations of a small sampling of possible partition zeta functions having particularly pleasing formulas.

\subsubsection{Zeta functions for partitions with parts restricted by congruence conditions}

%%%%2

Our first line of study will concern sets $\mathcal M\subset\mathbb N$ that are defined by congruence conditions. Note by \eqref{DirichletProduct} that for disjoint sets $\mathcal M_1$ and $\mathcal M_2$,
\[
\zeta_{\mathcal P_{\mathcal M_1\cup\mathcal M_2}}(s)
=
\zeta_{\mathcal P_{\mathcal M_1}}(s)\zeta_{\mathcal P_{\mathcal M_2}}(s)
.
\]
Hence, to study any set of partitions determined by congruence conditions on the parts, it suffices to consider series of the form
\[
\zeta_{\mathcal P_{a+m\mathbb N}}(s)
,
\]
where $a\in\Z_{\geq0}$, $m\in\N$  (excluding the case $a=0$, $m=1$, where the zeta function clearly diverges), and $\mathcal P_{a+m\mathbb N}$ is partitions into parts congruent to $a$ modulo $m$. We see examples of the case $\zeta_{\mathcal P_{m\mathbb N}}(2^N)=\zeta_{\mathcal P_{0+m\mathbb N}}(2^N)$ in \eqref{X} and \eqref{XX}; we are interested in the most general case, with $s=n\in\mathbb N$.   

%{\bf LR: Now we cite a formula from the previous section here}
%{\bf ***RPS: Insert description of easy cases from previous zeta paper as an equation.***}. {\bf LR: I think this isn't needed. Just put a few simple cases of this formula in the last section, and then cite them here, saying which specializations they correspond to for us, and then segue to say we have the more general formula}

Our first main result is then the following, where $\Gamma$ is the usual gamma function of Euler and $e(x):=e^{2\pi ix}$. The proof will use an elegant and useful formula highlighted by Chamberland and Straub in \cite{ChamberlandStraub}, which we note was also inspired by previous work on multiplicative partitions in \cite{ChamberlandJohnsonNadeauWu}. In fact, the following result is a generalization of Theorem 8 of \cite{ChamberlandJohnsonNadeauWu} which in our notation corresponds to $a=m=1$. 

\begin{theorem}\label{mainthm}
For $n\geq2$, we have
\[
\zeta_{\mathcal P_{a+m\mathbb N}}(n)
=\Gamma(1+a/m)^{-n}\prod_{r=0}^{n-1}\Gamma\left(1+\frac{a-e(r/n)}{m}\right)
.
\]
\end{theorem}
Theorem \ref{mainthm} has several applications. 
By considering the expansion of the logarithm of the gamma function, we easily obtain the following result, in which $\gamma$ is the Euler-Mascheroni constant and the principal branch of the logarithm is taken.

\begin{corollary}\label{FirstCor}
For any $m,n\geq2$, we have that
\begin{equation*}
\begin{aligned}
\log\left(\zeta_{\mathcal P_{a+m\mathbb N}}(n)\right)
&
=
n\log(1+a/m)+\frac{a(n+1)}m(1-\gamma)
-
\sum_{r=0}^{n-1}
\log\left(
1+\frac{a-e(r/n)}m
\right)
\\
&
+
\sum_{r=0}^{n-1}
\sum_{k\geq2}\frac{(-1)^k(\zeta(k)-1)\left(a^k+\left(a-e(r/n)\right)^k\right)}{km^k}
.
\end{aligned}
\end{equation*}
\end{corollary}

When $a=0$ and $m\geq2$, we obtain the following strikingly simple formula, which is similar to Theorem 7 of \cite{ChamberlandJohnsonNadeauWu} that in our notation corresponds to the case $a=m=1$.
\begin{corollary}\label{SecondCor}
For any $m,n\geq2$, we have that
\[
\log\left(\zeta_{\mathcal P_{m\mathbb N}}(n)\right)
=
n
\sum_{\substack{k\geq2\\ n|k}}\frac{\zeta(k)}{km^{k}}
.
\]
\end{corollary}

%{\bf RPS: Here there needs some transitional wording, re-ordered text follows, which we still want to evaluate for inclusion.} 

%%%%4

\subsubsection{Connections to ordinary Riemann zeta values}
%{\bf LR: Renamed subsection}
%{\bf LR: Added a few sentences}

In addition to providing interesting formulas for values of more exotic partition-theoretic zeta functions, the above results also lead to curious formulas for the classical Riemann zeta function. In fact, $\zeta(s)$ 
is itself a partition zeta function, summed over prime partitions, so it is perhaps not too surprising to find that we can learn something about it from a partition-theoretic perspective. Then we continue the theme of evaluations by recording a few results expressing the value of $\zeta$ at integer argument $n>1$ in terms of gamma factors. %, independently of multiplicative partitions. 

In the first, curious identity, let $\mu$ denote the classical M\"obius function. We point out that this is essentially a generalization of a formula for the case $a=m=1$ given in Equation 11 of \cite{ChamberlandJohnsonNadeauWu}. 

\begin{corollary}\label{FourthCor}
For all $m,n\geq2$, we have
\[
\zeta(n)=m^n\sum_{\substack{k\geq 1}}\frac{\mu(k)}k\sum_{r=0}^{nk-1}\log\left(\Gamma\left(1-\frac{e\left(\frac r{nk}\right)}m\right)\right)
.
\]
\end{corollary}

%{\bf RPS: I think we questioned whether to keep the following three formulas or not. Let's leave them in for now, see what Ken thinks. LR: agreeed.}

The next identity gives $\zeta(n)$ in terms of the $n$th derivative of a product of gamma functions.  The authors were not able to find this formula in the literature; however, given the well-known connections between $\Gamma$ and $\zeta$, as well as the known example below the following theorem, it is possible that the identity is known.
\begin{theorem}\label{RobertCor}
For integers $n>1$, we have
\[
\zeta(n)=\frac{1}{n!} \lim_{z\to 0^+} \frac{\mathrm d^n}{\mathrm d z^n}\prod_{j=0}^{n-1}\Gamma\left(1-ze(j/n)\right)
.
\]
\end{theorem}

\begin{example}As an example of implementing the above identity, take $n=2$; then using Euler's well-known product formula for the sine function, it is easy to check that
\[
\zeta(2)=\frac{1}{2!} \lim_{z\to 0^+} \frac{\mathrm d^2}{\mathrm d z^2}\Gamma\left(1+z\right)\Gamma\left(1-z\right)
=\frac{1}{2!} \lim_{z\to 0^+} \frac{\mathrm d^2}{\mathrm d z^2}\frac{\pi z}{\sin(\pi z)}=\frac{\pi^2}{6}
.
\]
\end{example} 

This last formula for $\zeta(n)$, following from a formula in \cite{Robert} together with the preceding theorem, is analogous to some extent to the classical identity $\sin(n)=\frac{e^{in}-e^{-in}}{2i}$.
\begin{corollary}\label{RobertCor2}
For integers $n>1$, we have
\[
\zeta(n)=\lim_{z\to 0^+}\frac{\prod_{j=0}^{n-1}\Gamma\left(1-ze(j/n)\right)-\prod_{j=0}^{n-1}\Gamma\left(1-ze(j/n)\right)^{-1}}{2z^n}
.
\]
\end{corollary}
%%%%%%%%%%%

%%%%5
\subsubsection{Zeta functions for partitions of fixed length}

We now consider zeta sums of the shape $\zeta_{\mathcal P}(\{s\}^k)$ as in Definition \ref{11}. Our first aim will be to extend \eqref{pzv} above (which is Corollary 2.4 of \cite{Robert}).  %; in that work, the second author defined the partition zeta function {\bf RPS: Is this even necessary? Maybe just use the case P'=P below? LR: I don't understand the question. We already defined some of this notation above, so if new notation isn't needed, make sure not to redefine, or to say that we are recalling the definition.}
%\[
%\zeta_{\mathcal P'}(\{s\}^k):=\sum_{\substack{\lambda\in\mathcal P'\\ \ell(\lambda)=k}}n_{\lambda}^{-s}, \mathrm{Re}(s)>1
%,
%\]
%taken over partitions in some subset $\mathcal P'$ of $\mathcal P$, where for a general partition $\lambda$, we will denote by $\ell(\lambda)$ the {\it length} of $\lambda$, i.e., the number of parts in $\lambda$.
%
%Here we will consider such sums taken over all partitions
%\[
%\zeta_{\mathcal P}(\{s\}^k)=\sum_{\substack{\lambda\in\mathcal P\\ \ell(\lambda)=k}}n_{\lambda}^{-s}.
%\]
%{\bf LR: Think already defined}
As we shall see below, these zeta values are special cases of sums considered by Hoffman in \cite{Hoffman1}.

%{\bf ***RPS: Need to give a couple of nice examples from zeta paper here.*** LR: Cite from section 2.}

Let $[z^n]f$ represent the coefficient of $z^n$ in a power series $f$. Using this notation, we show the following, which in particular gives an algorithmic way to compute each $\zeta_{\mathcal P}(\{m\}^k)$ in terms of Riemann zeta values for $m\in\mathbb N_{\geq 2}$.
\begin{theorem}\label{SecondTheorem}
For all $m\geq2$, $k\in\N$, we have
\begin{align*}
\zeta_{\mathcal P}(\{m\}^k)
&=
\pi^{mk}[z^{mk}]\prod_{r=0}^{m-1}\Gamma\left(1-\frac{z}{\pi}e(r/m)\right)\  \\
&=
\pi^{mk}[z^{mk}]\operatorname{exp}\left(\sum_{j\geq1}\frac{\zeta(mj)}{j}\left(\frac{z}{\pi}\right)^{mj}\right)
.
\end{align*}
\end{theorem}
%%%%%%%%%%%%%

Generalizing the comments just below \eqref{pzv}, the next corollary follows directly from Theorem \ref{SecondTheorem} (using the well-known fact that $\zeta(k)\in\mathbb Q\pi^k$ for even integers $k$).

\begin{corollary}\label{RationalityCor}
For $m\in 2\N$ even, we have that
\[
\zeta_{\mathcal P}(\{m\}^k)\in\mathbb Q\pi^{mk}
.
\]
\end{corollary}
\begin{remark}
This can also be deduced from Theorem 2.1 of \cite{Hoffman1}.
\end{remark}

%%%%7

To conclude this section, we note one explicit method for computing the values $\zeta_{\mathcal P}(\{m\}^k)$ at integral $k,m$ (especially if $m$ is even, in which case the zeta values below are completely elementary).

\begin{corollary}\label{DeterminantCor}
For $m\geq2, k\in\N$, and $j\geq i$, 
set 
\[
\alpha_{i,j}
:=
\zeta(m(j-i+1))\frac{(k-i)!}{\pi^{m(j-i+1)}(k-j)!}
.\]
Then we have 
\[
\zeta_{\mathcal P}(\{m\}^k)
=
\frac{\pi^{mk}}{k!} \det
\begin{pmatrix} 
\alpha_{1,1}&\alpha_{1,2}&\alpha_{1,3}&\ldots&\alpha_{1,k}
\\
-1 & \alpha_{2,2} &\alpha_{2,3} &\ldots & \alpha_{2,k}
\\
0&-1& \alpha_{3,3}&\ldots & \alpha_{3,k}
\\
\vdots &\vdots&\vdots &\ddots &\vdots
\\
0&0&\ldots&-1&\alpha_{k,k} 
\end{pmatrix}
.
\]
\end{corollary}

\begin{remark}
There are results resembling these in Knopfmacher and Mays \cite{KM}.
\end{remark}

\subsection{Analytic continuation and $p$-adic continuity}

%%%%3

%\begin{remark}
%Corollary \ref{FirstCor} and Corollary \ref{SecondCor} extend results for $(a,m)=(0,1)$ given in \cite{ChamberlandJohnsonNadeauWu}.
%end{remark}

%{\bf LR: All is new. Robert, please check carefully RPS: Will do, Larry, but deleting this comment for clean-ness.}

If we jump forward about 100 years from the pathbreaking work of Euler concerning special values of the Riemann zeta function at %positive 
even %(and even negative) 
integers, we arrive at the famous work of Riemann in connection with prime number theory. Namely, in 1859, Riemann brilliantly described the most significant properties of $\zeta(s)$ following that of an Euler product: the analytic continuation and functional equation for $\zeta(s)$. It is for this reason, of course, that the zeta function is named after Riemann, and not Euler, who had studied this function in some detail, and even conjectured a related functional equation. In particular, this analytic continuation allowed Riemann to bring the zeta function, and indeed the relatively new field of complex analysis, to the forefront of number theory by connecting its roots to the distribution of prime numbers.

It is natural therefore, whenever one is faced with new zeta functions, to ask about their prospect for analytic continuation. Here, we offer a brief study of some of these properties, in particular showing that the situation for our zeta functions is much more singular. Partition-theoretic zeta functions in fact naturally give rise to functions with essential singularities. Here, we use
Corollary \ref{SecondCor}
to study the continuation properties of partition zeta functions over partitions $\mathcal P_{m\mathbb N}$ into multiples of $m>1$. In order to state the result we first define, for any $\varepsilon >0$, the right half-plane $\mathbb H_{\varepsilon}:=\{z\in\C : \operatorname{Re}(z)>\varepsilon\}$, and we denote by $\frac1{\N}$ the set $\{1/n : n\in\N\}$.
\begin{corollary}\label{ThirdCor}
For any $\varepsilon>0$ and $m>1$, $\zeta_{\mathcal P_{m\mathbb N}}(s)$ has a meromorphic extension to $\mathbb H_{\varepsilon}$ with poles exactly at $\mathbb H_{\varepsilon}\cap\frac1{\N}$. In particular, there is no analytic continuation beyond the right half-plane $\operatorname{Re}(s)>0$, as there would be an essential singularity at $s=0$.
\end{corollary}
\begin{remark}For the function $\zeta_{\mathcal P_{\N}}(s)$, a related discussion of poles and analytic continuation was made by the user mohammad-83 in a mathoverflow.net question.\end{remark}
%%%%%%%

Finally, we follow Kubota and Leopoldt \cite{KL}, who showed $\zeta$ could be modified slightly to obtain modified zeta functions for any prime $p$ which extend $\zeta$ to the space of$p$-adic integers $\mathbb Z_p$, to obtain further examples of $p$-adic zeta functions of this sort. These continuations are based on the original observations of Kubota and Leopoldt, and, in a rather pleasant manner, on the evaluation formulas discussed above.

In particular, we will use Corollary \ref{DeterminantCor} to $p$-adically interpolate modified versions of $\zeta_{\mathcal P}(\{m\}^k)$ in the $m$-aspect. Given the connection discussed in Section \ref{MZVSection} to multiple zeta values, these results should be compared with the literature on $p$-adic multiple zeta values (e.g. see \cite{Furusho}), although we note that our $p$-adic interpolation procedure seems to be more direct in the special case we consider.  

The continuation in the $m$-aspect of this function is also quite natural, as the case $k=1$ is just that of the Riemann zeta function. Thus, it is natural to search for a suitable $p$-adic zeta function that specializes to the function of Kubota and Leopoldt when $k=1$. It is also desirable to find a $p$-adic interpolation result which makes the partition-theoretic perspective clear.

Here, we provide such an interpretation. Let us first denote the set of partitions with parts not divisible by $p$ as $\mathcal P_p$; then we consider the length-$k$ partition zeta values $\zeta_{\mathcal P_p}(\{s\}^k)$. Note that for $k=1$, $\zeta_{\mathcal P_p}(\{s\}^1)$ is just the Riemann zeta function with the Euler factor at $p$ removed (as considered by Kubota and Leopoldt). We then offer the following $p$-adic interpolation result.
\begin{theorem}\label{padicInterpThm}
Let $k\geq1$ be fixed, and let $p\geq k+3$ be a prime. Then $\zeta_{\mathcal P_p}(\{s\}^{k})$ can be extended to a continuous function for $s\in\Z_p$ which agrees with $\zeta_{\mathcal P_p}(\{s\}^{k})$ on a positive proportion of integers.
\end{theorem}

\subsection{Connections to multiple zeta values}\label{MZVSection}

%%%%8

%\subsection{Connections to multiple zeta values}
Our final application of the circle of ideas related to partition zeta functions and infinite products will be in the theory of multiple zeta values. 

\begin{definition}\label{mzvdef}
We define for natural numbers $m_1,m_2,\ldots ,m_k$ with $m_k>2$ the {\it multiple zeta value} (commonly written ``MZV'')
\begin{equation*}
\zeta(m_1,m_2,\ldots,m_k):=\sum_{n_1>n_2>\ldots>n_k\geq 1}\frac{1}{n_1^{m_1}\ldots n_k^{m_k}}
.
\end{equation*}
We call $k$ the {\it length} of the MZV. Furthermore, if $m_1=m_2=\ldots =m_k$ are all equal to some $m\in\mathbb N$, we use the common notation
\begin{equation}\label{mzvpzv}
\zeta(\{m\}^k):=\sum_{n_1>n_2>\ldots>n_k\geq 1}\frac{1}{\left(n_1 n_2\ldots n_k\right)^{m}}.
\end{equation}
\end{definition}

Multiple zeta values have a rich history and enjoy widespread connections; interested readers are referred to Zagier's short note \cite{ZagierMZV}, %Hoffman's webpage \cite{Hoffman2}, 
and for a more detailed treatment, the excellent lecture notes of Borwein and Zudilin \cite{BorweinZudilin}. There are many nice closed-form identities in the literature; for example, one can show (see \cite{Hoffman1}) on analogy to \eqref{pzv} that   
\begin{equation}
\zeta(\{2\}^k) = \frac{\pi^{2k}}{(2k+1)!}.
\end{equation}
Similar (but more complicated) expressions for $\zeta(\{2^t\}^k)$ for all $t\in\mathbb N$ are given in \cite{Robert}, parallel to those mentioned in Section \ref{Partition-theoretic zeta functions} for $\zeta_{\mathcal P}(\{2^t \}^k)$.  

Observe that the partition zeta function $\zeta_{\mathcal P}(\{m\}^k)$ (see Definition \ref{pzv}) can be rewritten in a similar-looking form to \eqref{mzvpzv} above:
\begin{equation}\label{mzvpzv2}
\zeta_{\mathcal P}(\{m\}^k)=\sum_{n_1\geq n_2\geq \ldots\geq n_k\geq 1}\frac{1}{\left(n_1 n_2\ldots n_k\right)^{m}}
\end{equation}
In fact, if we take $\mathcal P^*$ to denote partitions into distinct parts, then \eqref{mzvpzv} reveals $\zeta(\{m\}^k)$ is equal to the partition zeta function $\zeta_{\mathcal P^*}(\{m\}^k)$ summed over length-$k$ partitions into distinct parts, as pointed out in \cite{Robert}. Series such as those in \eqref{mzvpzv2} have been considered and studied extensively by Hoffman (for instance, see \cite{Hoffman1}). 

%
%{\bf ***RPS: This section should include one or two easy MZV examples from previous paper.*** LR: I think now just cite equaiton or two from section 2.}

By reorganizing sums of the shape \eqref{mzvpzv2}, we arrive at interesting relations between $\zeta_{\mathcal P}(\{m\}^k)$ and families of MZVs. In order to describe these relations, we first recall that a {\it composition} is simply a finite tuple of natural numbers, and we call the sum of these integers the {\it size} of the composition. Denote the set of all compositions by $\mathcal C$ and write $|\lambda|=k$ for $\lambda=(a_1,a_2,\ldots,a_j)\in\mathcal C$ if $k=a_1+a_2+\ldots+a_j$. Then we obtain the following. 

%{\bf RPS: Switched $n$ to $m$ in the arguments of the MZV on the RHS below... was n , should be m as a factor of each argument, right?}

\begin{proposition}\label{DecouplingCorollary}
Assuming the notation above, we have that
\[
\zeta_{\mathcal P}(\{m\}^k)
=
\sum_{\substack{\lambda=(a_1,\ldots, a_j)\in\mathcal C\\ |\lambda|=k}}\zeta(a_1m,a_2m,\ldots,a_jm)
.
\]
\end{proposition}
\begin{remark}
Proposition \ref{DecouplingCorollary} is analogous to results of Hoffman; the reader is referred to Theorem 2.1 of \cite{Hoffman1}.
\end{remark}

In particular, for any $n>1$ we can find the following reduction of $\zeta(\{n\}^k)$ to MZVs of smaller length. We note in passing that Theorem 2.1 of \cite{Hoffman1} also shows directly how to write these values in terms of products (as opposed to simply linear combinations) of ordinary Riemann zeta values: hints, perhaps, of further connections.
We remark in passing that this can be thought of as a sort of ``parity result'' (cf. \cite{IKZ,Tsumura}).
\begin{corollary}\label{ParallelMZVLowerOrders}
For any $n,k>1$, the MZV $\zeta(\{n\}^k)$ of length $k$ can be written as an explicit linear combination of MZVs of lengths less than $k$.
\end{corollary}

As our final result, we give a simple formula for $\zeta(\{n\}^k)$. This formula is probably already known; if $k=2$ it follows from a well-known result of Euler (see the discussion of $H(n)$ on page 3 of \cite{Zagier}) and is closely related to (11) and (32) of \cite{BorweinBradley}. The idea of the proof is also similar to what has appeared in, for example, \cite{Zagier}. However, the authors have decided to include it due to connections with the ideas used throughout this paper, and the simple deduction of the formula from expressions necessary for the proofs of the results described above. 

\begin{proposition}\label{ParallelValuesExp}
The MZV $\zeta(\{n\}^k)$ of length $k$ can be expressed as a linear combination of products of ordinary $\zeta$ values. In particular, we have
\[
\zeta(\{n\}^k)
=
(-1)^k
\left[z^{nk}\right]
\operatorname{exp}
\left(-
\sum_{j\geq1}
\frac{\zeta(nj)}jz^{nj}
\right)
.
\]
\end{proposition}
%{\bf LR: Minus sign fixed here LR: Can this go? It looks like there should be another set of parentheses.}
\begin{remark}
This formula is equivalent to a special case of Theorem 2.1 of \cite{Hoffman1}. However, since the proof is very simple and ties in with the other ideas in this paper, we give a proof for the reader's convenience.
\end{remark} 

The proof of Corollary \ref{DeterminantCor} yields a similar determinant formula here.
\begin{corollary}\label{DeterminantCor2}
For $n\geq2, k\in\N$, and $j\geq i$, 
set 
\[
\beta_{i,j}
:=
-
\zeta(n(j-i+1))\frac{(k-i)!}{(k-j)!}
.\]
Then we have 
\[
\zeta(\{n\}^k)
=
\frac{(-1)^k}{k!} \det
\begin{pmatrix} 
\beta_{1,1}&\beta_{1,2}&\beta_{1,3}&\ldots&\beta_{1,k}
\\
-1 & \beta_{2,2} &\beta_{2,3} &\ldots & \beta_{2,k}
\\
0&-1& \beta_{3,3}&\ldots & \beta_{3,k}
\\
\vdots &\vdots&\vdots &\ddots &\vdots
\\
0&0&\ldots&-1&\beta_{k,k} 
\end{pmatrix}
.
\]
\end{corollary}
%{\bf LR: Minus signs fixed here too LR: Can this go now?}
\begin{remark} We can see from the above corollary that $\zeta(\{n\}^k)$ is a linear combination of products of zeta values, which is closely related to formulas of Hoffman \cite{Hoffman1}.
\end{remark}

\subsection{Machinery}

\subsubsection{Useful formulas}
In this section, we collect several formulas that will be key to the proofs of the theorems above.  We begin with the following beautiful formula given by Chamberland and Straub in Theorem 1.1 of \cite{ChamberlandStraub}). In fact, this formula has a long history, going back at least to Section 12.13 of \cite{WhitttakerWatson}, and we note that Ding, Feng, and Liu independently discovered this same result in Lemma 7 of \cite{DingFengLiu}.
\begin{theorem}\label{Armin11}
If $n\in\N$ and $\alpha_1,\ldots,\alpha_n$ and $\beta_1,\ldots,\beta_n$ are complex numbers, none of which are non-positive integers, with $\sum_{j=1}^n\alpha_j=\sum_{j=1}^n\beta_j$, then we have
\[
\prod_{k\geq0}
\prod_{j=1}^n\frac{(k+\alpha_j)}{(k+\beta_j)}=\prod_{j=1}^n\frac{\Gamma(\beta_j)}{\Gamma(\alpha_j)}
.
\]
\end{theorem}

We will also require two Taylor series expansions for $\log\Gamma$, both of which follow easily from Euler's product definition of the gamma function \cite{Edwards}. The first expansion, known as Legendre's series, is valid for $|z|<1$ (see (17) of \cite{Wrench}):

\begin{equation}\label{LogGammaSecondFormula}
\log\Gamma(1+z)
=
-\gamma z+\sum_{k\geq2}\frac{\zeta(k)}k(-z)^k
.
\end{equation}
We also have the following expansion valid for $|z|<2$ (see (5.7.3) of \cite{Nist}):
\begin{equation}\label{573Nist}
\log\Gamma(1+z)=-\log(1+z)+z(1-\gamma)+\sum_{k\geq2}(-1)^k(\zeta(k)-1)\frac{z^k}k
.
\end{equation}

Furthermore, we need a couple of facts about Bell polynomials (see Chapter 12.3 of \cite{Andrews}). The $n$th \emph{complete Bell polynomial} is the sum
\[
B_n(x_1,\dots,x_n):=\sum_{i=1}^n B_{n,i}(x_1,x_2,\dots,x_{n-i+1})
.
\]
The $i$th term here is the polynomial
\[
B_{n,i}(x_1,x_2,\dots,x_{n-i+1}):=\sum \frac{n!}{j_1!j_2!\cdots j_{n-i+1}!}
\left(\frac{x_1}{1!}\right)^{j_1}\left(\frac{x_2}{2!}\right)^{j_2}\cdots\left(\frac{x_{n-i+1}}{(n-i+1)!}\right)^{j_{n-i+1}},
\]
where we sum over all sequences $j_1, j_2,..., j_{n-i+1}$ of nonnegative integers such that $j_1+j_2+\cdots+j_{n-i+1}=i$ and $j_1+2j_2+3j_3+\cdots+(n-i+1)j_{n-i+1}=n$.

With these notations, we use a specialization of the classical Fa\`{a} di Bruno formula \cite{diBruno}, which allows us to write the exponential of a formal power series as a power series with coefficients related to complete Bell polynomials: 
\begin{equation}\label{FdB}
\operatorname{exp}\left(\sum_{j=1}^\infty \frac{a_j}{j!} x^j \right)
= \sum_{k=0}^\infty \frac{B_k(a_1,\dots,a_k)}{k!} x^k
\end{equation}

Fa\`{a} di Bruno also gives an identity \cite{diBruno} that specializes to the following formula for the $k$th complete Bell polynomial in the series above as the determinant of a certain $k\times k$ matrix: 

\begin{equation}\label{determinant}
B_k(a_1,\dots,a_k)
=
\det
\left(
\begin{smallmatrix} 
a_1&\binom{k-1}{1}a_2&\binom{k-1}{2}a_3&\binom{k-1}{3}a_4&\ldots&\ldots&a_k
\\\\
-1 & a_1 &\binom{k-2}{1}a_2&\binom{k-2}{2}a_3&\ldots&\ldots&a_{k-1}
\\\\
0&-1& a_1&\binom{k-3}{1}a_2&\ldots&\ldots& a_{k-2}
\\\\
0&0&-1& a_1&\ldots&\ldots& a_{k-3}
\\\\
0&0&0&-1&\ldots&\ldots& a_{k-4}
\\\\
\vdots &\vdots&\vdots &\vdots &\ddots &\ddots &\vdots
\\\\
0&0&0&0&\ldots&-1&a_1 
\end{smallmatrix}
\right)
\end{equation}

\subsubsection{Proofs of Theorems \ref{mainthm} and \ref{RobertCor}, and their corollaries}

We begin with the proof of our first main formula.
\begin{proof}[Proof of Theorem \ref{mainthm}]
By \eqref{DirichletProduct}, we find that 
\[
\zeta_{\mathcal P_{a+m\mathbb N}}(n)
=
\prod_{k\in a+m\N}\frac{k^n}{k^n-1}=\prod_{j\geq1}\frac{(a+mj)^n}{(a+mj)^n-1}
=
\prod_{j\geq0}\prod_{r=0}^{n-1}\frac{(j+1+a/m)^n}{\left(j+1+\frac{a-e(r/n)}{m}\right)}
.
\]
Using Theorem \ref{Armin11}  and the well-known fact that
\begin{equation}\label{ExpSumZero}
\sum_{j=0}^{n-1}e(j/n)=0
\end{equation} directly gives the desired result.
\end{proof}
\begin{proof}[Proof of Corollary \ref{FirstCor}]
For this, we apply \eqref{573Nist}
and use \eqref{ExpSumZero},
 the obvious fact that 
 \[|(a-e(j/n))/m|<2
 ,
 \]
and the easily-checked fact that \[1+(a-e(j/n))\] is never a negative real number for $j=0,\ldots,n-1$. 
\end{proof}
\begin{proof}[Proof of Corollary \ref{SecondCor}]
Here, we simply use \eqref{LogGammaSecondFormula}.
Again, the corollary is proved following a short, elementary computation, using the classical fact that 
\[
\sum_{r=0}^{n-1}e(rk/n)=\begin{cases}n&\text{ if }n|k,\\ 0&\text{ else.}\end{cases}
\]
\end{proof}
\begin{proof}[Proof of Corollary \ref{ThirdCor}]
By Corollary \ref{SecondCor}, we find for $n\geq2$ that
\[
\log\left(\zeta_{\mathcal P_{m\mathbb N}}(n)\right)
=
\sum_{k\geq1}\frac{\zeta(nk)}{km^{kn}}
.
\]
Suppose that $\operatorname{Re}(s)>0$ and $s\not\in\frac1{\N}$.  Then letting 
\[
K
:=
\max
\{
\lceil1/\operatorname{Re}(s)\rceil+1,\operatorname{Re}(s)
\},
\]
it clearly suffices to show that 
\[\sum_{k\geq K}\frac{\zeta(sk)}{km^{ks}}\]
converges. But in this range on $k$, by choice we have $\operatorname{Re}(sk)>1$, so that using the assumption $m\geq2$, we find for $\operatorname{Re}(s)>0$ the upper bound
\begin{align*}
\sum_{k\geq K}\frac{\zeta(sk)}{km^{ks}}
&\leq \zeta(Ks)\sum_{k\geq K}\frac1{k2^{k\operatorname{Re}(s)}}
\leq \zeta(Ks)\sum_{k\geq 1}\frac1{k2^{k\operatorname{Re}(s)}}\\
&=
-\zeta(Ks)\log\left(2^{-\operatorname{Re}(s)}\left(2^{\operatorname{Re}(s)}-1\right)\right)
,
\end{align*}
and note that in the argument of the logarithm in the last step, by assumption we have $2^{\operatorname{Re}(s)}-1>0$. 

Conversely, if $s\in\frac1{\N}$, then it is clear that this representation shows there is a pole of the extended partition zeta function, as one of the terms gives a multiple of $\zeta(1)$.
\end{proof}
\begin{proof}[Proof of Corollary \ref{FourthCor}]
We utilize a variant of M\"obius inversion, reversing the order of summation in the double sum $\sum_{k\geq1}\sum_{d|k}\mu(d)f(nk)k^{-s}$; if
\[
g(n)=\sum_{k\geq1}\frac{f(kn)}{k^s}
,
\]
then 
\[
f(n)=\sum_{k\geq1}\frac{\mu(k)g(kn)}{k^s}
.
\]
Applying this inversion procedure to Corollary \ref{SecondCor}, so that $g(n)=\log\zeta_{\mathcal P_{m\mathbb N}}(n)$ (taking $s=1$), and $f(n)=\zeta(n)/m^n$, we directly find that 
\[
\zeta(n)=m^n\sum_{\substack{k\geq1}}\frac{\mu(k)}k\log\left(\zeta_{\mathcal P_{m\mathbb N}}(nk)\right)
.
\]
Applying Theorem \ref{mainthm} then gives the result.
\end{proof}

\begin{proof}[Proof of Theorem \ref{RobertCor}]
By the comments following Theorem 1.1 in \cite{Robert}, for $\mathcal M\in\mathbb N$ we have 
\begin{equation}\label{ProofEq}
\prod_{k\in\mathcal M}\left(1-\frac{z^s}{k^s}\right)^{-1}=1+z^s\sum_{k\in\mathcal M}\frac{1}{k^{s}\prod_{\substack{j\in\mathcal M\\ j\leq k}}\left(1-\frac{z^s}{j^s}\right)}
;
\end{equation}
thus
\[
\sum_{k\in\mathcal M}k^{-s}=\lim_{z\to 0^+}\frac{\prod_{k\in\mathcal M}\left(1-\frac{z^s}{k^s}\right)^{-1}-1}{z^s}
.\]
Taking $\mathcal M=\mathbb N, s=n\in\Z_{\geq 2}$, we apply L'Hospital's rule $n$ times to evaluate the limit on the right-hand side. The theorem then follows by noting, from Theorem \ref{Armin11}, that in fact
\[
\prod_{k\in\mathbb N}\left(1-\frac{z^n}{k^n}\right)^{-1}=\prod_{j=0}^{n-1}\Gamma\left(1-ze(j/n)\right)
.\] 
\end{proof}

\begin{proof}[Proof of Corollary \ref{RobertCor2}]
Picking up from the proof of Theorem \ref{RobertCor} above, it follows also from Theorem 1.1 of \cite{Robert} that
\[
\prod_{k\in\mathcal M}\left(1-\frac{z^s}{k^s}\right)=1-z^s\sum_{k\in\mathcal M}\frac{\prod_{\substack{j\in\mathcal M\\ j<k}}\left(1-\frac{z^s}{j^s}\right)}{k^{s}}
.
\]
Subtracting this equation from \eqref{ProofEq}, making the substitutions $\mathcal M=\mathbb N$, $s=n\geq 2$ as in the proof above, and using Theorem \ref{Armin11}, gives the corollary.
\end{proof}

\subsubsection{Proof of Theorem \ref{SecondTheorem} and its corollaries}
\begin{proof}[Proof of Theorem \ref{SecondTheorem}]
Using a similar method as in \cite{Robert} and a similar rewriting to that used in the proof of Theorem \ref{mainthm}, we note that a short elementary computation shows
\[
\sum_{k\geq0}\frac{z^{mk}}{\pi^{mk}}\zeta_{\mathcal P}(\{m\}^k)=\prod_{k\geq1}\frac{1}{1-\frac{z^m}{\pi^mk^m}}=\prod_{k\geq0}\prod_{r=0}^{m-1}\frac{(k+1)^m}{\left(k+1-\frac{z}{\pi}e(r/m)\right)}
.
\]
Much as in the proof of Theorem \ref{RobertCor}, using Theorem \ref{Armin11}, we directly find that this is equal to 
\[
\prod_{r=0}^{m-1}\Gamma\left(1-\frac{z}{\pi}e(r/m)\right)
,
\]
which gives the first equality in the theorem. 
Applying Equation \eqref{LogGammaSecondFormula} (formally we require $|z|<\pi$, but we are only interested in formal power series here anyway), we find immediately, using a very similar calculation to that in the proof of Corollary \ref{SecondCor}, that
\begin{equation}\label{PowSer1}
\begin{aligned}
\sum_{k\geq0}\left(\frac{z}{\pi}\right)^{mk}\zeta_{\mathcal P}(\{m\}^k)
&
=
\operatorname{exp}\left(\sum_{r=0}^{m-1}\sum_{j\geq2}\frac{\zeta(j)}j\left(\frac{z}{\pi}\right)^{mj}e(rj/m)\right)
\\
&
=
\operatorname{exp}\left(m\sum_{\substack{j\geq2\\ m|j}}\frac{\zeta(j)}j\left(\frac{z}{\pi}\right)^{j}\right)
,
\end{aligned}
\end{equation}
which is equivalent to the second equality in the theorem.
\end{proof}

\begin{proof}[Proof of Corollary \ref{DeterminantCor}]
Replace $x$ with $z^m$ in Equation \ref{FdB}, and set
\[
a_j=\frac{(j-1)!\zeta(mj)}{\pi^{mj}} 
\]
on the left-hand side (which becomes the right-hand side of \eqref{PowSer1}). Then comparing the right side of \ref{FdB} to the left side of \ref{PowSer1}, we deduce that
\[
\zeta_{\mathcal P}(\{m\}^k)=\frac{\pi^{mk}}{k!}B_k(a_1,\dots,a_k)
.\]
To complete the proof, we substitute the determinant in \ref{determinant} for $B_k(a_1,\dots,a_k)$ and rewrite the terms in the upper half of the resulting matrix as $\alpha_{i,j}$, as defined in the statement of the corollary.  
\end{proof}
\begin{proof}[Proof of Theorem \ref{padicInterpThm}]
In analogy with the calculation of Theorem \ref{SecondTheorem}, we find that
\[
\sum_{k\geq0}z^{mk}\zeta_{\mathcal P_p}(\{m\}^k)=\prod_{\substack{k\geq1\\ p\nmid k}}\frac{1}{1-\frac{z^m}{k^m}}=\frac{\prod_{k\geq0}\prod_{r=0}^{m-1}\frac{(k+1)^m}{\left(k+1-ze(r/m)\right)}}{\prod_{k\geq0}\prod_{r=0}^{m-1}\frac{(k+1)^m}{\left(k+1-\frac{z}{p}e(r/m)\right)}
}
=\prod_{r=0}^{m-1}\frac{\Gamma\left(1-ze(r/m)\right)}{\Gamma\left(1-\frac{z}{p}e(r/m)\right)}
.
\]
As in the calculation of \eqref{PowSer1}, this is equal to 
\[
\operatorname{exp}\left(\sum_{j\geq1}\frac{\zeta(mj)}{j}\left(z\right)^{mj}\left(1-1/p^{mj}\right)\right)
,
\]
so if we set
\[
\alpha_{i,j}^{(p)}(m)
:=
\zeta^*(m(j-i+1))\frac{(k-i)!}{(k-j)!}
,\]
where 
\[
\zeta^*(s):=(1-p^{-s})\zeta(s),
\]
then we have 
\[
\zeta_{\mathcal P_p}(\{m\}^k)
=
\frac{1}{k!} \det
\begin{pmatrix} 
\alpha^{(p)}_{1,1}&\alpha^{(p)}_{1,2}&\alpha^{(p)}_{1,3}&\ldots&\alpha^{(p)}_{1,k}
\\
-1 & \alpha^{(p)}_{2,2} &\alpha^{(p)}_{2,3} &\ldots & \alpha^{(p)}_{2,k}
\\
0&-1& \alpha^{(p)}_{3,3}&\ldots & \alpha^{(p)}_{3,k}
\\
\vdots &\vdots&\vdots &\ddots &\vdots
\\
0&0&\ldots&-1&\alpha^{(p)}_{k,k} 
\end{pmatrix}
.
\]
We further define $\zeta_{\mathcal P_p}(\{m\}^k)$ for more general values in $\mathbb C$, such as $m\in-\mathbb N$ using the analytic continuation of $\zeta$ in each of the factors $\alpha_{i,j}^{(p)}(m)$.
Next we recall the Kummer congruences, which state that if $k_1,k_2$ are positive even integers not divisible by $(p-1)$ and $k_1\equiv k_2\pmod{p^{a+1}-p^a}$ for $a\in\mathbb N$ where $p>2$ is prime, then 
\[
\left(1-p^{k_1-1}\right)\frac{B_{k_1}}{k_1}\equiv\left(1-p^{k_2-1}\right)\frac{B_{k_2}}{k_2}\pmod{p^{a+1}}
.
\]
Let us take $S_{s_0}$ to be the set of natural numbers congruent to $s_0$ modulo $p-1$.
The Kummer congruences then imply that for any $s_0\not\equiv0\pmod{p-1}$, and for any $k_1,k_2\in S_{s_0}$ with $k_1\equiv k_2\pmod{p^a}$ and $k_1,k_2>1$, that
\[
\zeta^*(1-k_1)\equiv\zeta^*(1-k_2)\pmod{p^{a+1}}
.
\]
If we choose $m_1,m_2\in S_{s_0}$ with $m_1\equiv m_2\pmod{p^a}$, then the values $1-(1-m_1)(j-i+1)$, $1-(1-m_2)(j-i+1)$ are in $S_{1+(s_0-1)(j-i-1)}$ and are congruent modulo $p^a$, and as $p>k$ the additional factorial terms (inside and outside the determinant) are $p$-integral. Now in our determinant, $j-i+1$ ranges through $\{1,2,\ldots, k\}$, and we want to find an $s_0$ such that $1+(s_0-1)r\not\equiv0\pmod{p-1}$ for $r\in\{1,2,\ldots k\}$.
 If we take $s_0=2$, then the largest value of $1+(s_0-1)r$ is $k+1$, which is by assumption less than $p-1$, and hence not divisible by it. Hence, in our case, $s_0=2$ suffices.
Thus, if $m_1,m_2\in S_{2}$ with $m_1\equiv m_2\pmod{p^a}$, then 
\[
\zeta_{\mathcal P_p}(\{1-m_1\}^k)\equiv\zeta_{\mathcal P_p}(\{1-m_2\}^k)\pmod{p^{a+1}}
.
\]
This shows that our zeta function is uniformly continuous on $S_2$ in the $p$-adic topology. As this set is dense in $\Z_p$, we have shown that the function extends in the $m$-aspect to $\Z_p$.

\end{proof}

\subsubsection{Proofs of results concerning multiple zeta values}
\begin{proof}[Proof of Proposition \ref{DecouplingCorollary}]
Recall from \eqref{mzvpzv2} 
that we need to study the sum
\[
\sum_{n_1\geq n_2 \geq \ldots \geq n_k\geq1}\frac1{(n_1n_2\ldots n_k)^m}
.
\]
The proof is essentially combinatorial accounting, keeping track of the number of ways to split up a sum
\[
\sum_{n_1\geq n_2 \geq \ldots \geq n_k\geq1}
\]
over all all $k$-tuples of natural numbers into a chain of equalities and strict inequalities. Suppose that we have 
\[
n_{1}\geq n_{2}\geq\ldots\geq n_{k}\geq1
.
\]
Then if any of these inequalities is an equality, say $n_{j}=n_{j+1}$, in the contribution to the sum 
\[
\sum_{n_1\geq n_2\geq \ldots\geq n_k\geq 1}(n_1\ldots n_k)^{-m}
,
\]
the terms $n_{j}$ and $n_{j+1}$ ``double up''. That is, we can delete the $n_{j+1}$ and replace the $n_{j}^{-m}$ in the sum with a $n_{j}^{-2m}$. Thus, the reader will find that our goal is to keep track of different orderings of $>$ and $=$, taking symmetries into account. The possible chains of $=$ and $>$ are encoded by the set of compositions of size $k$, by associating to the composition $(a_1,\ldots,a_j)$ the chain of inequalities
\[
n_{1}=\ldots=n_{a_1}>n_{a_1+1}=\ldots=n_{a_1+a_2}>n_{a_2+1}>\ldots >n_{k}
.
\]
That is, the number $a_1$ determines the number of initial terms on the right which are equal before the first inequality, $a_2$ counts the number of equalities in the next block of inequalities, and so on. It is clear that the sum corresponding to the each composition then contributes the desired amount to the partition zeta value in the corollary.
\end{proof}

\begin{proof}[Proof of Corollary \ref{ParallelMZVLowerOrders}]
In Proposition \ref{DecouplingCorollary}, comparison with Corollary \ref{DeterminantCor} shows that we have a linear relation among MZVs and products of zeta values. Observe that in $\zeta_{\mathcal P}(\{m\}^k)$, the only composition of length $k$ is $(1,1,\ldots,1)$, which contributes $k! \zeta(\{m\}^k)$ to the right-hand side of Proposition \ref{DecouplingCorollary}, and that the rest of the compositions are of lower length, hence giving MZVs of smaller length; the corollary follows immediately.
\end{proof}

\begin{proof}[Proof of Proposition \ref{ParallelValuesExp}]
Consider the multiple zeta value $\zeta(\{n\}^k)$ of length $k$. Then we directly compute
\[
\sum_{k\geq0}(-1)^k\zeta(\{n\}^k)z^{nk}
=
\prod_{m\geq1}
\left(
1-\left(\frac zm\right)^n
\right)
=
\prod_{m\geq0}\prod_{r=0}^{n-1}\frac{(m+1-ze(r/n))}{(m+1)^n}
.
\]
By Theorem \ref{Armin11}, this equals 
\[
\prod_{r=0}^{n-1}\Gamma(1-ze(r/n))^{-1}
.
\] 
Using precisely the same computation as was made in the proof of Theorem \ref{SecondTheorem}, we find that this is equal to
\[
\operatorname{exp}\left(-n\sum_{\substack{j\geq2\\ n|j}}\frac{\zeta(j)}jz^{j}\right)
.
\]
Hence, we have that
\[
\zeta(\{n\}^k)
=(-1)^k
\left[z^{nk}\right]
\operatorname{exp}
\left(-
\sum_{j\geq1}
\frac{\zeta(nj)}jz^{nj}
\right)
.
\]
\end{proof}
\begin{proof}[Proof of Corollary \ref{DeterminantCor2}]
Here we proceed exactly as in the proof of Corollary \ref{DeterminantCor}, except we make the simpler substitution
\[
a_k=(k-1)!\zeta(nk) 
\] 
into Equation \ref{FdB}, and compare with Proposition \ref{ParallelValuesExp}. In the final step, we replace the terms in the upper half of the matrix with $\beta_{i,j}$ as defined in the statement of the corollary.  
\end{proof}

\subsection{Some further thoughts}

We have presented samples of a few varieties of flora one finds at the fertile intersection of combinatorics and analysis. What unifies all of these is the perspective that they represent instances of partition zeta functions, with proofs that fit naturally into the Eulerian theory we propound. 

We close this article by noting a general class of partition-theoretic analogs of classical Dirichlet series having the form
\[
\mathcal D_{\mathcal P'}(f,s):=\sum_{\lambda\in\mathcal P'}f(\lambda)n_{\lambda}^{-s},
\]
where $\mathcal P'$ is a proper subset of $\mathcal P$ and $f : \mathcal P' \rightarrow \C$. Of course, partition zeta functions arise from the specialization $f\equiv 1$, just as in the classical case. 

Taking $\mathcal P'=\mathcal P_{\mathcal M}$ as defined previously, then if $f:=f(n_{\lambda})$ is completely multiplicative with appropriate growth conditions, it follows from Theorem 1.1 of \cite{Robert} that we have a generalization of \eqref{DirichletProduct}
\begin{equation}\label{DirichletProduct2}
\mathcal D_{\mathcal \mathcal P_{\mathcal M}}(f,s)=\prod_{j\in \mathcal M}\left(1-\frac{f(j)}{j^s}\right)^{-1}\  \left(\mathrm{Re}(s)>1\right),
\end{equation}
and nearly the entire theory of partition zeta functions noted here (and developed in \cite{Robert}) extends to these series as well. 

\section*{Acknowledgements}
The authors thank Armin Straub for useful discussion on the history of Theorem \ref{Armin11}. 
\section{Zeta polynomials}

We now turn our attention toward a different layer of connections between partitions and zeta functions, via the theory of modular forms. Although Euler's generating function for $p(n)$ is essentially modular, and Euler also anticipated the study of $L$-functions that are intimately tied to modular forms, the true depth of such observations did not come into view until further work on complex analysis was carried out in the nineteenth century. 

It turns out that all modular forms are related to partitions in a very direct way. Here we recall the case of this connection for modular forms for the full modular group $\textrm{SL}_2(\Z)$. We then use these modular forms to define the second class of functions that are the topic of this paper, the {\it zeta polynomials} associated to modular forms.

\subsection{Partitions and modular forms}

In a paper from 2004  \cite{BKO}, the first author, Bruinier, and Kohnen
investigated the values of
a certain sequence of modular functions in connection with
the arithmetic properties of meromorphic modular forms on
$\textrm{SL}_2(\Z)$. One of the main results in the paper shows that integer partitions and
a universal sequence of polynomials encode the Fourier
expansions of modular forms.

Here we recall this result, which can be thought of as a precise formulation of the assertion that
a modular form is distinguished by its ``first few coefficients." In fact, by making use of partitions we have
an effective recursive procedure which computes the Fourier coefficients in order.

We take $q:=e^{2\pi i z}$ throughout this section. Now, suppose that
$$
f(z)=\sum_{n=h}^{\infty}a_f(n)q^n
$$
is a weight $k \in 2\Z$ meromorphic modular form on $\textrm{SL}_2(\Z)$.
If $k\geq 2$
is even, then let $E_k(z)$ denote the normalized Eisenstein series
\begin{equation}
  E_{k}(z):=1-\frac{2k}{B_k}\sum_{n=1}^{\infty}\sigma_{k-1}(n)q^n.
\end{equation}
Here $B_k$ denotes the  usual $k$th Bernoulli number and
$\sigma_{k-1}(n):=\sum_{d\mid n}d^{k-1}$. If $k>2$,
then $E_k(z)$ is a weight $k$ modular form on $\textrm{SL}_2(\Z)$; although
the Eisenstein series
\begin{equation}
E_2(z)=1-24\sum_{n=1}^{\infty}\sigma_1(n)q^n
\end{equation}
is not a modular form, it also plays an important role.
As usual, let $j(z)$ denote the modular function on $\textrm{SL}_2(\Z)$ which is
holomorphic on $\mathbb{H}$, the upper half of the complex plane, with Fourier expansion
$$
  j(z):=q^{-1}+744+196884q+21493760q^2+\cdots. 
$$
We will require a specific sequence of modular functions $j_m(z)$; to
define this sequence, we set
\begin{equation}\label{1.4}
j_0(z):=1\ \ \ \ \ \ {\text {\rm and}}\ \ \
\ \ \ j_1(z):=j(z)-744. 
\end{equation}
If $m\geq 2$, then define $j_m(z)$ by
\begin{equation}\label{1.5}
j_m(z):= j_1(z) \ |
\ T_0(m), 
\end{equation}
where $T_0(m)$ is the usual normalized $m$th weight
zero Hecke operator.
Each $j_m(z)$ is a monic polynomial in $j(z)$ of degree $m$.
Here we list the first few:
\begin{displaymath}
\begin{split}
&j_0(z)=1,\\
&j_1(z)=j(z)-744=q^{-1}+196884q+\cdots,\\
&j_2(z)=j(z)^2-1488j(z)+159768=q^{-2}+42987520q+\cdots,\\
&j_3(z)=j(z)^3-2232j(z)^2+1069956j(z)-36866976=q^{-3}+2592899910q+\cdots.
\end{split}
\end{displaymath}

Let $\frak{F}$ denote the usual fundamental domain of the action
of $\textrm{SL}_2(\Z)$ on $\mathbb{H}$. By assumption, $\frak{F}$ does not
include the cusp at $\infty$. Throughout, let $i:=\sqrt{-1}$ and
let $\omega:=(1+\sqrt{-3})/2$. If $\tau\in \frak{F}$, then define
$e_{\tau}$ by
\begin{equation}\label{1.7}
e_{\tau}:=\begin{cases} 1/2 \ \ \ \ \ &{\text {\rm if}}\ \tau=i,\\
                 1/3 \ \ \ \ \ &{\text {\rm if}}\ \tau=\omega,\\
                   1 \ \ \ \ \ &{\text {\rm otherwise}}.
\end{cases}
\end{equation}

By studying the logarithmic derivatives of modular forms, the first author, Bruinier and Kohnen obtained the 
 next result in \cite{BKO}, which offers universal polynomial recursion formulas for the Fourier coefficients of modular forms.

\begin{theorem}\label{BKOThm1} For every positive integer $n$
define
$F_n(x_1,\dots,x_{n})\in \Q[x_1,\dots,x_{n}]$ by
\begin{displaymath}
\begin{split}
&F_n(x_1,\dots,x_{n})
:=-\frac{2x_1\sigma_1(n)}{n}\\
& \ +\sum_{\substack{m_1,\dots,m_{n-1}\geq 0,\\
m_1+2m_2+\cdots+(n-1)m_{n-1}=n}}
(-1)^{m_1+\cdots+m_{n-1}}\cdot
\frac{(m_1+\cdots+m_{n-1}-1)!}{m_1!\cdots m_{n-1}!} \cdot
x_2^{m_1}\cdots x_{n}^{m_{n-1}}.
\end{split}
\end{displaymath}
If $ f(z)=q^h+\sum_{n=1}^{\infty}a_f(h+n)q^{h+n}$ is a weight $k$ meromorphic
modular  form on $\textrm{SL}_2(\Z)$, then for every positive integer $n$ we
have
$$a_{f}(h+n)=F_n(k,a_{f}(h+1),\dots,a_{f}(h+n-1))-\frac{1}{n}
\sum_{\tau\in \frak{F}}e_{\tau} \ord_{\tau}(f)\cdot
j_{n}(\tau).$$
\end{theorem}

\begin{remark} 
Theorem~\ref{BKOThm1} illustrates a deep connection between integer partitions and the coefficients of modular forms. An inspection of the summation in the theorem reveals the role of partitions: the partitions of $n$, apart from $(n)$ itself, determine
the recurrence formula for $a_f(h+n)$. We note that for each such partition
\[
\lambda=\left( (n-1)^{m_{n-1}}(n-2)^{m_{n-2}}...2^{m_2} 1^{m_1} \right)\vdash n
\] 
with $m_k$ being the multiplicity of $k$ as a part, %with $m_k\geq 0$ being the multiplicity of $k$ as a part of $\lambda$, 
we can rewrite $m_1+m_2+...+m_{n-1}=\ell(\lambda)$.
\end{remark}

\noindent
The first few polynomials $F_n$ are
\begin{displaymath}
\begin{split}
&F_1(x_1):=-2x_1,\\
&F_2(x_1,x_2):=-3x_1+\frac{x_2^2}{2},\\
&F_3(x_1,x_2,x_3):=-\frac{8x_1}{3}-\frac{x_2^3}{3}+x_2x_3,\\
&F_4(x_1,x_2,x_3,x_4):=-\frac{7x_1}{2}-x_2^2x_3+x_2x_4+\frac{x_2^4}{4}
+\frac{x_3^2}{2}.\end{split}
\end{displaymath}
The $n=1$ case of Theorem 1
implies that
$$
a_f(h+1)=60k-744h-\sum_{\tau\in \frak{F}}e_{\tau}\ord_{\tau}(f)\cdot
j(\tau).
$$
\medskip
\noindent
{\bf Example 1.}
Since
$$\Delta(z)=\sum_{n=1}^{\infty}\tau(n)q^n=q-24q^2+252q^3-\cdots,
$$
the unique normalized
weight 12 cusp form on $\textrm{SL}_2(\Z)$,
is nonvanishing in $\mathbb{H}$,
Theorem~\ref{BKOThm1} implies that
$$
\tau(n+1)=F_n(12,\tau(2),\dots,\tau(n)).
$$

\subsection{Manin's zeta polynomials}\label{ManinZetaPolySection} Theorem~\ref{BKOThm1} shows that integer partitions and universal polynomials play central roles in computing the Fourier coefficients of modular forms. The only additional data required is a form's leading coefficient, weight, and its divisor. In the spirit of the previous section, it is then natural to ask
for a theory of zeta functions related to partitions, which somehow arises from this connection. Here we address this problem by giving a brief exposition of recent
work on a problem of Manin on zeta polynomials.

%{\bf LR: Rest of this section is new}

To begin, we consider any newform $f=\sum_{n\geq1}a_nq^n\in S_k(\Gamma_0(N))$ of even weight $k$ and level $N$. The standard zeta-type function associated to $f$ is its $L$-function
\[
L(f,s):=\sum_{n\geq1}a_nq^n
,
\]
which may be normalized so that the completed $L$-function 
$$
\Lambda(f,s) := \Big(\frac{\sqrt{N}}{2\pi}\Big)^s \Gamma(s) L(f,s)  
$$ 
may be analytically continued and satisfies the functional equation 
\[
\Lambda(f,s) = \epsilon(f) \Lambda(f,k-s)
,
\]
with $\epsilon(f)= \pm 1$.  The {\it critical $L$-values} are the complex numbers
$L(f,1)$, $L(f,2)$, $\dots$, $L(f,k-1)$.  Manin's conjecture then states that these critical $L$-values can be suitably packaged to fit into the following framework.

\begin{definition}[Manin]\label{ManinZetaPolyDefn} A polynomial $Z(s)$ is a {\bf zeta-polynomial} if it satisfies the following criteria: 
\begin{enumerate}
\item (Naturality) It is arithmetic-geometric in origin.

\smallskip
\item (Functional Equation) For $s\in \C$ we have $Z(s)=\pm Z(1-s)$.

\smallskip
\item (Riemann Hypothesis) If $Z(\rho)=0$, then $\re(\rho)=1/2$. 

\smallskip
\item The values $Z(-n)$ have a ``nice'' generating function 

\smallskip
\item The values $Z(-n)$ encode arithmetic-geometric information.

\end{enumerate}
\end{definition}

We remark that, of course, it is very important that the arithmetic-geometric nature of the function in part (1) of the definition is satisfied. For example, any polynomial with real coefficients and satisfying the Riemann Hypothesis automatically satisfies the functional equation in part (2), and hence the compelling arithmetic nature of the particular object being constructed is critical. Part (4) is also central to Manin's idea, and is meant to parallel Euler's power series expansion for the Riemann zeta function:
\begin{equation}\label{BN}
\frac{t}{1-e^{-t}}=1+\frac12 t-t\sum_{n=1}^{\infty}{\zeta(-n)}\cdot \frac{t^n}{n!}
\end{equation}
This generating function \eqref{BN} for the values $\zeta(-n)$ also has a well-known interpretation in $K$-theory \cite{FG}, in line with part (5) of the above definition. Namely, it is essentially the generation function
    for the torsion of the $K$-groups for $\Q$. 

Recently in \cite{ORS}, Sprung and the first two authors confirmed Manin's speculations and offered the following resolution to his question. To describe this result, first consider the $m$th {\it weighted moments} of critical values:
\begin{equation}\label{Moments}
M_f(m):=\sum_{j=0}^{k-2}\left(\frac{\sqrt N}{2\pi}\right)^{j+1}\frac{L(f,j+1)}{(k-2-j)!}j^m=\frac1{(k-2)!}\sum_{j=0}^{k-2}\binom{k-2}j\Lambda(f,j+1)j^m.
\end{equation}
For positive integers $n$, we recall the usual generating function for
the (signed) Stirling numbers of the first kind: %{\bf LR: Trying to stick with colons for consistency; we should check throughout}
\begin{equation}
(x)_n=x(x-1)(x-2)\cdots (x-n+1)=:\sum_{m=0}^n s(n,m)x^m
\end{equation}
Using these numbers, the zeta-function satisfying Manin's definition is given by 
\begin{equation}\label{Zdef}
Z_f(s):=\sum_{h=0}^{k-2} (-s)^h \sum_{m=0}^{k-2-h} \binom{m+h}{h}\cdot s(k-2,m+h)\cdot M_f(m).
\end{equation}
The main result of \cite{ORS} confirms that this natural object satisfies a zeta-type functional equation and the Riemann Hypothesis.
\begin{theorem}[Theorem 1.1 of \cite{ORS}]\label{Thm1} If $f\in S_k(\Gamma_0(N))$ is an even weight $k\geq 4$ newform, then the following are true:
\begin{enumerate}
\item For all $s\in \C$ we have that $Z_f(s)=\epsilon(f)Z_f(1-s)$.
\item If $Z_f(\rho)=0$, then $\re(\rho)=1/2$.
\end{enumerate}
\end{theorem}

In accordance with Definition \ref{ManinZetaPolyDefn}, we also want to find natural interpretations of the values $Z_f(-n)$ at negative integers and for their generating functions. 
This may be accomplished by considering the ``Rodriguez-Villegas Transform''  of  \cite{RV}. Theorem~\ref{Thm1}  is naturally related to the arithmetic of period polynomials\footnote{This is a slight renormalization of the period polynomials considered in
references such as \cite{CPZ, KZ, PP, Z}.}
\begin{equation}
R_f(z):=\sum_{j=0}^{k-2} \binom{k-2}{j}\cdot \Lambda(f,k-1-j)\cdot z^j.
\end{equation}
The values of $Z_f(s)$ at non-positive integers are then the coefficients expanded around $z=0$ of the rational function $$\frac{R_f(z)}{(1-z)^{k-1}}.$$
The following result is the analogue of \eqref{BN} for the Riemann zeta function, and is our answer to part (4) of the above definition.

\begin{theorem}[Theorem 1.3 of \cite{ORS}]\label{Thm2} Assuming the notations and hypotheses above, as a power series in $z$ we have
$$
\frac{R_f(z)}{(1-z)^{k-1}}=\sum_{n=0}^{\infty} Z_f(-n)z^n.
$$
\end{theorem}

In accordance with the aforementioned $K$-theoretic interpretation of the values of $\zeta(s)$ at negative integers, and in relation to part (5) of Definition \ref{ManinZetaPolyDefn},
 it is natural to ask whether the $z$-series in Theorem \ref{Thm2} 
    has an analogous interpretation---that is, what (if any) arithmetic information is encoded by the values $Z_f(-n)$?
Hints along these lines were first offered by Manin in \cite{M}, where he produced similar zeta-polynomials by applying the Rodriguez-Villegas transform to the odd period polynomials for Hecke eigenforms on $\operatorname{SL}_2(\Z)$ studied by
Conrey, Farmer, and Imamo$\mathrm{\bar{g}}$lu \cite{CFI}. He asked for a generalization for the full period polynomials
for such Hecke eigenforms in connection to recent work of El-Guindy and Raji \cite{ER}. Theorems \ref{Thm1} and \ref{Thm2} answer this question and
provide the generalization 
for all even weight $k\geq 4$ newforms on congruence subgroups of the form $\Gamma_0(N)$. Theorem~\ref{Thm1} additionally offers an explicit combinatorial description of the zeta-polynomials in terms of weighted moments. In fact, the developments of \cite{ORS} are made possible by the recent general proof of the following Riemann Hypothesis for period polynomials.
\begin{theorem}[\cite{RHPP}]\label{RHPP}
For any even integer $k\geq 4$, and any level $N$, all of the zeros of the period polynomial $R_f(z)$ are on the unit circle.  
\end{theorem}
\begin{remark}
The reason that the term ``Riemann Hypothesis'' is fitting for this result is that the period polynomials satisfy a natural functional equation relating the values at $z$ and $-1/z$ (itself arising from the functional equation for $L(f,s)$), and this result states that all zeros of the period polynomials lie on the line of symmetry for their functional equations.
\end{remark}
\begin{remark}\label{PreciseRmk}
In fact, the authors of \cite{RHPP} prove a somewhat stronger statement, giving rather precise bounds on the location of the angles of these zeros.
\end{remark}

In line with part (5) of Definition \ref{ManinZetaPolyDefn}, we offer a conjectural combinatorial arithmetic-geometric interpretation of the $Z_f(s)$.
To this end, we make use of the Bloch-Kato Conjecture, which offers a Galois cohomological interpretation for critical values of motivic $L$-functions
\cite{BK}. Here we consider the special case of the critical values $L(f,1), L(f,2),\dots, L(f,k-1)$.
These conjectures are concerned with motives $\mathcal{M}_f$ associated to $f$, but the data needed for this conjecture can be found in the {\it $\lambda$-adic realization} $V_\lambda$ of $\mathcal{M}_f$ for a prime $\lambda$ of $\Q(f)$, where $\Q(f)$ is the field generated by the Hecke eigenvalues $a_n(f)$ (where we have $a_1(f)=1$). The Galois representation $V_\lambda$ associated to $f$ is due to Deligne, and we recall the essential properties below. For a high-brow construction of $V_\lambda$ from $\mathcal{M}_f$, we refer to the seminal paper of Scholl \cite{scholl}.

Deligne's theorem says that for a prime $\lambda$ of $O_{\Q(f)}$ lying above $l$, there is a continuous linear representation $V_\lambda$ unramified outside $lN$
$$ \rho_{f,\lambda}: \Gal(\overline{\Q}/\Q)\rightarrow \GL(V_\lambda),$$
 so that for a prime $p \nmid lN$, the arithmetic Frobenius $\Frob_p$ satisfies 
 $$\Tr(\rho_f(\Frob_p^{-1}))=a_p(f), \text{ and } \det(\rho_f(\Frob_p^{-1}))=p^{k-1}.$$
 
We may also consider the $j$th Tate twist $V_\lambda(j)$, which is $V_\lambda$ but with the action of Frobenius multiplied by $p^j$.  After choosing a $ \Gal(\overline{\Q}/\Q)$-stable lattice $T_\lambda$ in $V_\lambda$, we may consider the short exact sequence $$ 0 \longrightarrow T_\lambda(j) \longrightarrow V_\lambda(j)\stackrel{\pi}\longrightarrow V_\lambda/T_\lambda(j)\longrightarrow 0.$$
Bloch and Kato define local conditions $H^1_\mathbf{f}(\Q_p,V_\lambda(j))$ for each prime $p$. We let $H^1_\mathbf{f}(\Q,V_\lambda(j))$ be the corresponding global object, i.e., the elements of $H^1(\Q,V_\lambda(j))$ whose restriction at $p$ lies in $H^1_\mathbf{f}(\Q_p,V_\lambda(j))$. Analogously, we may define $H^1_\mathbf{f}(\Q,V_\lambda/T_\lambda(j))$, which is the Bloch--Kato $\lambda$- Selmer group. The \v{S}afarevi\v{c}--Tate group is
$$\Sha_f(j)=\Directsum_\lambda \frac{H^1_\mathbf{f}(\Q,V_\lambda/T_\lambda(j))}{\pi_*H^1_\mathbf{f}(\Q,V_\lambda(j))}.$$
The Bloch--Kato Tamagawa number conjecture then claims the following.

\begin{conjecture}[Bloch--Kato]\label{bkt} Let $0\leq j\leq k-2$, and assume $L(f,j+1)\neq0$. Then we have
$$\frac{L(f,j+1)}{(2\pi i)^{j+1} \Omega^{(-1)^{j+1}}}=u_{j+1} \times \frac{\Tam(j+1) \# \Sha(j+1)}{\#H^0_\Q(j+1)\#H^0_\Q(k-1-j)}=:C(j+1). $$
\end{conjecture}

Here, $\Omega^\pm$ denotes the Deligne period, $\Tam$ deontes the product of the Tamagawa numbers, $H^0_\Q$ is the set of global points, and $u_{j+1}$ is a non-specified unit of $\Q(f)$.

\begin{remark}
Note that $L(f,j+1)\neq0$ in this range provided that $j+1\neq k/2$. 
\end{remark}
We denote the normalized version of $C(j+1)$ by 
\begin{equation}
\widetilde{C(j+1)}=C(j+1)\cdot \frac{(i\sqrt{N})^{j+1}\Omega^{(-1)^{j+1}}}{(k-2-j)!},
\end{equation}
but when $L(f,j+1)=0$, we define $\widetilde{C(j+1)}:=0$. 

\begin{theorem}[Theorem 1.4 of \cite{ORS}]\label{Thm3} Assuming the Bloch-Kato Conjecture and the notation  above, we have that
 %For each $1\leq j\leq k-1$ we have that $$L(f,j)={\text {\rm algebraic interpretation}}.
 %$$
% In particular we have that
 $$
 M_f(m)=\sum_{0\leq j\leq k-2} \widetilde{C(j+1)}j^m,
 $$
 which in turn implies for each non-negative integer $n$ that
 $$
 Z_f(-n)= \sum_{j = 0}^{k-2}\left(\sum_{h=0}^{ k-2}\sum_{m = 0}^{k-2-h} n^h  \binom{m+h}{h}\cdot s(k-2,m+h)\right)j^m\widetilde{C(j+1)}.
 $$
\end{theorem}

Finally, we offer an intriguing connection of the zeta polynomials $Z_f(s)$ to the world of combinatorial geometry. 
This investigation is motivated by the following question.
In light of the vast importance of the distribution of the zeros of the Riemann zeta function on (presumably) the critical line, it is natural to consider the distribution of the zeros of $Z_f(s)$ on the line $\re(s)=1/2$. 
Of course, the $Z_f(s)$ are polynomials, and so any interesting distributional properties only make sense in the context of limits of sequences of modular forms as the level or weights go to infinity. In this sense, one can ask: do their zeros behave in a manner which is analogous to
the zeros of the Riemann zeta-function $\zeta(s)$? Namely, how are their zeros distributed in comparison with the growth of
$$N(t):= \# \{  \rho=s+it\ : \ \zeta(\rho)=0 \ \ {\text {\rm with}}\ 0<t\leq T\}, 
$$
which is well known to satisfy
\begin{equation}\label{NT}
N(T)=\frac{T}{2\pi} \log\frac{T}{2\pi} -\frac{T}{2\pi}+O(\log T)?
\end{equation}
As we shall see, the zeros of $Z_f(s)$ 
behave in a manner that is somewhat analogous to (\ref{NT}) in terms of its highest zero.

To make this precise, it turns out that the correct answer lies in comparing $Z_f(s)$ with two families of combinatorial polynomials.
In what follows, we note that for $x,y\in\C$, the binomial coefficient $\binom xy$ is defined by 
\[
\binom{x}{y}:=\frac{\Gamma(x+1)}{\Gamma(y+1)\Gamma(x-y+1)}.
\]
We see below that the $Z_f(s)$, depending on $\epsilon(f)$, can naturally be estimated by the polynomials
\begin{equation}
H_{k}^{+}(s):=\binom{s+k-2}{k-2}+\binom{s}{k-2}
,
\end{equation}
\begin{equation}
H_{k}^{-}(s):=\sum_{j=0}^{k-3}\binom{s-j+k-3}{k-3}.
\end{equation}

\begin{theorem}[Theorem 1.3 of \cite{ORS}]\label{ZeroDistribution}
Assuming the notations and hypotheses above, the following are true: 
\begin{enumerate}
\item The zeros of $H^{-}_{k}(-s)$ lie on the line $\re(s)=1/2$,  and they are the complex numbers $\rho=\frac{1}{2}+it$ 
with $t\in \mathbb R$ such that the value of the monotonically decreasing function 
\[
h_k(t):=\sum_{j=0}^{k-3}\cot^{-1}\left(\frac{2t}{2j+1}\right)
\]
lies in the set $\{\pi,2\pi,\ldots,(k-3)\pi \}$. Similarly, the zeros of $H^{+}_{k}(-s)$ lie on the line $\re(s)=1/2$ and have imaginary parts $t$ which may be found by solving for $h_k(t)$ to lie in the set $\{\pi/2,3\pi/2,\ldots,(k-5/2)\pi \}$.
Moreover, as $k\rightarrow\infty$, the highest pair of complex conjugate roots of $H^{-}_k(s)$ have imaginary part equal in absolute value to
\[
\frac{(k-3)(k-1)}{2\pi}+O(1)
,
\]
and the height of the highest roots of $H^+_k(s)$ is 
\[
\frac{(k-3)(k-1)}{\pi}+O(1)
.
\]
\item Let $f\in S_4(\Gamma_0(N))$ be a newform. If $\epsilon(f)=-1$, then the only root of $Z_f(s)$ is at $s=1/2$. If $\epsilon(f)=1$, then there are two roots of $Z_f(s)$, and as $N\rightarrow\infty$, their roots converge on the sixth order roots of unity $\operatorname{exp}(\pm\pi i /3)$. 
\item For fixed $k\geq6$, as $N\rightarrow +\infty$, the zeros of $Z_f(s)$ for newforms $f\in S_k(\Gamma_0(N))$ with $\epsilon(f)=\pm 1$ converge to the zeros of $H^{\pm}_{k}(-s)$. 
Moreover, for all $k,N$, if $\epsilon(f)=1$ {\text {\rm (resp. $\epsilon(f)=-1$)}}, then the imaginary part of the largest root is strictly bounded by 
$(k-3)\left(k-\frac 72\right)$  {\text {\rm (resp. 
$(k-4)\left(k-\frac 92\right)$)}}.
\end{enumerate}
\end{theorem}

\begin{remark}
Theorem~\ref{ZeroDistribution} (3) is somewhat analogous to (\ref{NT}). Since the zeros of $Z_f(s)$ are approximated by those of  $H^{\pm}_{k}(-s)$, the analog of $N(T)$ is dictated by Theorem~\ref{ZeroDistribution} (1), where the largest zero
has imaginary part $\sim \frac{k^2}{2\pi}$ or $\sim  \frac{k^2}{\pi}$ depending on the sign of the functional equation. Numerical investigations indicate that the locations of the ``high'' zeros are however somewhat differently arranged than those of the Riemann zeta function, the $n$th highest pair of zeros being at approximately height $N_0/n$, where $N_0$ is the highest zero.
\end{remark}

As we shall describe shortly, the zeta polynomials $Z_f(s)$ may be thought of,  via Theorem \ref{ZeroDistribution}, as arithmetic-geometric Ehrhart polynomials. In this comparison, the combinatorial structure in Theorem~\ref{Thm3}, which we call
the ``Bloch-Kato complex'', serves as an analogue of a polytope. 
Assuming the Bloch-Kato Conjecture, Theorem~\ref{Thm3} describes the values $Z_f(-n)$ as combinatorial sums of $m$-weighted moments of the $j$th Bloch-Kato components. 
To describe this combinatorial structure, we made use of  the Stirling numbers $s(n,k)$ which can be arranged in a ``Pascal-type'' triangle 
$$
\begin{array}{ccccccccccccc}
 &  &   &        &      &       & 1 &  &   &  &  &  \\
& &   &       &      &  0     &  & 1 &   &  &  &   &  \\
&  &  &        & 0    &       & -1 & & 1 &  & &  & \\
& &  & 0      &     & 2       &    &  -3    &  & 1 &  &  &   \\
&  & 0 &       & -6  &      & 11  &         & -6 &  & 1 &  &  \\
  & 0 &   & 24    &    &-50     &    & 35      &   & -10 &   & 1 &  \\
0 \ \ &    & -120 &   & 274    &    & -225   &  & 85 &  & -15 &  & \ \ 1\\
\end{array}
$$
thanks to the recurrence relation
$$
s(n,k)=s(n-1,k-1)-(n-1)\cdot s(n-1,k).
$$
This follows from the obvious relation
$$
(x)_n=(x)_{n-1} (x-n+1)=x(x)_{n-1}-(n-1)(x)_{n-1}.
$$
The Bloch-Kato complex is then obtained by cobbling together weighted layers of these Pascal-type triangles using the binomial coefficients appearing in (\ref{Zdef}).

The connection to Ehrhart polynomials arises from the central role played by the $H_k^{\pm}(-s)$ in our study of the $Z_f(s)$.
In \cite{RV}, Rodriguez-Villegas proved that certain  Hilbert polynomials, such as the  $H^{\pm}_{k}(-s)$, which are Rodriguez-Villegas transforms of $x^{k-2}\pm1$, are examples of zeta-polynomials. 
%In related work, in analogy with the local zeta factors of Hecke $L$-series \`a la Tate's thesis, Bump, Choie, Kurlberg and Vaaler \cite{BCKV} studied %examples related to Hermite polynomials.  
These well-studied combinatorial polynomials  encode important geometric structure such as the distribution of integral points in polytopes.

Given a $d$-dimensional integral lattice polytope in $\mathbb R^n$, we recall that the Ehrhart polynomial $\mathcal L_p(x)$ is determined by 
\[
\mathcal L_p(m)=\#\left\{p\in\Z^n : p\in m\mathcal P\right\}
.
\]
The polynomials $H_{k}^-(s)$ whose behavior determines an estimate for those of $Z_f(s)$ (when $\epsilon(f)=-1)$  as per Theorem \ref{ZeroDistribution} are the Ehrhart polynomials of the simplex (cf. \cite{BHW})
 \[\operatorname{conv}\left\{e_1,e_2,\ldots,e_{k-3},-\sum_{j=1}^{k-3}e_j\right\}.
\]
We note that in Section 1.10 of \cite{GRV}, Gunnells and Rodriguez-Villegas also gave an enticing interpretation of the modular-type behavior of Ehrhart polynomials. Namely, they noted that the polytopes $P$ with vertices in a lattice $L$, when acted upon by $\operatorname{GL}(L)$ in the usual way, have a fixed Ehrhart polynomial for each equivalence class of polytopes. Hence, these classes may be thought of as points on a ``modular curve'', and the operation of taking the $\ell$th coefficient of the Ehrhart polynomial is analogous to a modular form. This analogy is strengthened as they define a natural Hecke operator on the set of Ehrhart polynomials, such that the $\ell$th coefficients of them are eigenfunctions. Moreover, they show that these eigenclasses are all related to explicit, simple Galois representations. Thus it is natural---as well as intriguing---to speculate about the relationship between these observations and Theorem \ref{ZeroDistribution}. In particular, we have shown that as the level $N$ of cusp forms of a fixed weight $k$ tends to infinity, the coefficients of the zeta-polynomial $Z_f(s)$ tend to (a multiple of) these coefficients of Ehrhart polynomials considered in \cite{GRV}. It is also interesting to note that Zagier defined (see \cite{ZagHecke}) a natural Hecke operator on the period polynomials of cusp forms, which commutes with the usual Hecke operators acting on cusp forms. Thus one may ask if there is a reasonable interpretation of Hecke operators on the zeta functions $Z_f(s)$, which ties together this circle of ideas.

In the following sections, we will give a brief outline of the main underlying ideas of the proofs of these results. In particular, we devote the next subsection to some tools needed for the proofs, follow up with the proof sketches, and finish with several examples.

\subsection{Tools needed for the proof of Manin's conjecture}\label{RHPPTools}
\subsubsection{Work of Rodriguez-Villegas}

We begin our summary of key tools for the proofs of the preceding section by recalling the clever and useful observations of Rodriguez-Villegas in \cite{RV}. We will only be concerned with a specialized result; although the reader is referred to the original paper for a more general version and interesting mathematical context of his work. We first take a polynomial $U(z)$ of degree $e$ with $U(1)\neq0$. The non-vanishing condition is non-essential, as otherwise one simply factors out all the powers of $1-z$ from the polynomial and then applies the Rodriguez-Villegas transform; however, it is convenient for the simplest description to have the exact degree of the polynomial with these factors pinned down. We then define the rational function
\[
P(z):=\frac{U(z)}{(1-z)^{e+1}}
.
\]
Expanding in $z$, we have
\[
P(z)=\sum_{n=0}^{\infty}h_nz^n
,
\]
and it is easy to see that that there is a polynomial $H(z)$ of degree $d-1$ such that $H(n)=h_n$ for all $n$. The key observation of \cite{RV} is then as follows.
\begin{theorem}[Rodriguez-Villegas]\label{RVThm}
If all roots of $U$ lie on the unit circle, then all roots of $Z(s)$ lie on the vertical line $\operatorname{Re}(z)=1/2$. Moreover, if $U$ has real coefficients and $U(1)\neq0$, then $Z(s)$ satisfies the functional equation
\[
Z(1-s)=(-1)^{d-1}Z(s)
.
\]
\end{theorem}
\begin{proof}
The first claim is simply the special case of the Theorem of \cite{RV} when $d=e+1$. The second claim was described in Section 4 of \cite{RV}, but for the reader's convenience we sketch the proof. By the single proposition of \cite{RV}, it suffices to show that $P(1/z)=(-1)^dx^{d-e}P(z)$. Now suppose that $U$ factors as
\[
U(z)=(z-\rho_1)\ldots(z-\rho_e)
,
\]
where each $\rho_j$ is on the unit circle but not equal to $1$. Then 
\[
z^eU\left(\frac 1z\right)=(1-z\rho_1)\ldots(1-z\rho_e).
\]
Since the coefficients of $U$ are real, we have $(-1)^e\rho_1\rho_2\ldots\rho_e=1$, and dividing by this quantity yields 
\[
U\left(\frac 1z\right)=z^{-e}U(z)
.
\]
The claimed transformation for $P$ then follows directly from the definition.

\end{proof}

\subsubsection{Period polynomials at $1$}
Here, we recall a simple result that we shall need, related to the order of $R_f(s)$ at $s=1$. 
 \begin{lemma}\label{ValueOneNotZeroPP}
Assuming the notation above, $R_f(1)\neq0$ if $\epsilon(f)=1$, and $R_f(s)$ has a simple zero at $s=1$ if $\epsilon(f)=-1$.
\end{lemma}
\begin{proof}
The functional equation for $\Lambda(f,s)$ shows that
\begin{align}\label{RF1Eqn}
R_f(1)
&=
\sum_{j=0}^{k-2} \binom{k-2}{j}\Lambda(f,j+1)\  
\nonumber \\
&=\begin{cases}
 \Lambda\left(f,\frac k2\right)+2\sum_{j=\frac{k+2}2}^{k-2}\binom{k-2}{j}\Lambda(f,j+1)\  &\text{ if }\epsilon(f)=1,
 \\
\Lambda\left(f,\frac k2\right)\  &\text{ if } \epsilon(f)=-1.
\end{cases}
\end{align}
Now $\Lambda(f,s)$ is real-valued on the real line, and
well-known work of Waldspurger \cite{W} implies that $\Lambda\left(f,\frac k2\right)\geq0$. Furthermore, Lemma 2.1 of \cite{RHPP} states that
\begin{equation}\label{ChainIneqCritVal}
0\leq\Lambda\left(f,\frac k2\right)\leq \Lambda\left(f,\frac k2+1\right)\leq\ldots\leq\Lambda(f,k-1),
\end{equation}
and that $\Lambda\left(f,\frac k2\right)=0$ if $\epsilon(f)=-1$. So, if $\epsilon(f)=1$, then the expression in the first case of \eqref{RF1Eqn} is composed of all non-negative terms, which cannot all vanish as it is impossible for all periods of $f$ to be zero. Hence, in this case, $R_f(1)\neq0$. If $\epsilon(f)=-1$, then as $\Lambda\left(f,\frac k2\right)=0$, we see that $R_f(1)=0$. To see that this zero is simple, note in a similar manner that all terms in $R_f'(1)$ are non-negative, with the last term being $(k-2)\Lambda(f,k-1)$. But this term cannot be zero, as the chain of inequalities in \eqref{ChainIneqCritVal} would then imply that all periods of $f$ are zero.
\end{proof}
\subsubsection{Ingredients for Theorem~\ref{Thm3}}

We first describe the local conditions $H^1_\mathbf{f}(\Q_p,V_\lambda(j))$ for a given prime $p$, following \cite[Section 3]{BK}. Recall that $\lambda$ was the prime above $l$ in Deligne's representation $V_\lambda$. 

The first case is when $p=l$. Here, we define $$H^1_\mathbf{f}(\Q_p,V_\lambda(j)):=\ker\left(H^1_\mathbf{f}(D_p,V_\lambda(j))\rightarrow H^1_\mathbf{f}(D_p,V_\lambda(j)\otimes\mathbb{B}_{\text{cris}})\right),$$
where $D_p$ denotes a decomposition group for a prime over $p$. For a definition of the $\Q_p$-algebra $\mathbb{B}_{\text{cris}}$, we refer to Berger's article \cite[II.3]{Berger}.

For the other cases (i.e. $p\neq l$), we let $$H^1_\mathbf{f}(\Q_p,V_\lambda(j)):=\ker\left(H^1_\mathbf{f}(D_p,V_\lambda(j))\rightarrow H^1_\mathbf{f}(I_p,V_\lambda(j))\right),$$
where $I_p$ is the inertia subgroup. We let $H^1_\mathbf{f}(\Q,V_\lambda(j))$ be the corresponding global object, i.e., the elements of $H^1(\Q,V_\lambda(j))$ whose restriction at $p$ lies in $H^1_\mathbf{f}(\Q_p,V_\lambda(j))$. 

We note that Bloch and Kato's Tamagawa number conjecture \ref{bkt} is independent of any choices, cf. \cite[Section 6]{D}, or for more detail cf. \cite[Proposition 5.14 (iii)]{BK} and \cite[page 376]{BK}, in which the independence of the choice in lattice in the Betti cohomology is discussed. 

Secondly, we describe the set of global points $H^0_\Q$, with the appropriate Tate twists:

$$ H^0_\Q(j):=\bigoplus_\lambda H^0(\Q,V_\lambda/T_\lambda(j))$$

\subsection{Sketch of proofs for theorems on zeta polynomials and the Riemann Hypothesis for period polynomials}\label{RHPPPfs}
Here we sketch the proofs of the theorems in Section \ref{ManinZetaPolySection}. We begin with the proof of the Riemann Hypothesis for period polynomials, which is an essential ingredient in the main proofs abou zeta polynomials.
\begin{proof}[Sketch of the proof of Theorem \ref{RHPP}]
The general strategy is to consider, for $m:=\frac{k-2}{2}$, the polynomial 
$$
P_f(X):=\frac{1}{2}\left(\begin{matrix} 2m\\ m \end{matrix}\right)\Lambda\left(f,\frac{k}{2}\right)+
\sum_{j=1}^m \left(\begin{matrix}2m\\ m+j \end{matrix}\right) \Lambda\left(f,\frac{k}{2}+j\right)X^j.
$$

The result then follows if the unit circle contains all of the zeros of
$$
T_f(X):=P_f(X)+\epsilon(f) P_f(1/X).
$$
Following a standard trick, we substitute $X \rightarrow z=e^{i\theta}$, note that $T_f(i\theta)$ takes on only real values,  and hope to find enough sign changes to locate all zeros using the Intermediate Value Theorem. In fact, it turns out that $T_f(z)$ is a trigonometric polynomial in $\sin$ or
$\cos$ depending on $\epsilon(f)$. For large values of $k$ or $N$, results of P{\' o}lya \cite{Polya} and Szeg{\" o} \cite{Szego} then give strong conditions on the locations of the zeros, given unimodal properties of the coefficients of these polynomials. Overall, this boils the argument down (in the limit) to proving several handy inequalities for the competed critical $L$ values of $f$. These are shown to follow from several arguments making keen use of the famous Hadamard factorization for $\Lambda(f,s)$ in terms of the roots of $\Lambda$.

This proof is only sufficient for modular forms of large weight or level. For the remaining cases, separate arguments can be directly constructed (for example, as we will look at in Section \ref{RHPPExamples}, if the weight is a small number like $4$, the polynomial is only of degree $2$ and can hence be studied somewhat directly). After this, finitely many cases remain, which may then be checked numerically, as the authors of \cite{RHPP} do using SAGE.
\end{proof}

We proceed to outline the proof of Theorem \ref{Thm1}, which gives the version of the Riemann Hypothesis inherited by $Z_f(s)$ from the Riemann Hypothesis for $P_f(s)$. 
Firstly, however, we need the result of Theorem \ref{Thm2}.
\begin{proof}[Proof of Theorem~\ref{Thm2}]
This proof is a combinatorial exercise; we refer the reader interested in the details to \cite{ORS}.
\end{proof}

\begin{proof}[Proof of Theorem \ref{Thm1}]
We first let
\[
\widehat R_f(z):=\frac{R_f(z)}{(1-z)^{\delta_{-1,\epsilon(f)}}}
,
\]
where $\delta_{i,j}$ is the Kronecker delta function. 
By Theorem \ref{RHPP} and Lemma \ref{ValueOneNotZeroPP}, we see that $\widehat R_f$ is a polynomial of degree $k-2-\delta_{-1,\epsilon(f)}$, all of whose roots lie on the unit circle, and such that $\widehat R_f(1)\neq0$.
Thus we have
\[
\frac{R_f(z)}{(1-z)^{k-1}}=\frac{\widehat R_f(z)}{(1-z)^{k-1-\delta_{-1,\epsilon(f)}}}
.
\]
Applying Theorems \ref{Thm2}, \ref{RHPP}, and \ref{RVThm} with $d=k-1-\delta_{-1,\epsilon(f)}$ yields the result, and in particular shows that the zeros of $Z_f(s)$ lie on the line $\re(s)=1/2$.

\end{proof}

\begin{proof}[Proof of Theorem~\ref{Thm3}]
The proof of Theorem~\ref{Thm3} follows immediately from replacing the terms involving $L(f,j+1)$ by $\widetilde{C(j+1)}$, and appropriate normalizations.
\end{proof}

\begin{proof}[Sketch of the proof of Theorem~\ref{ZeroDistribution}]

Part (1)  follows directly from the observation (which is easy to combinatorially verify) that the polynomials $H^-_{k}(x)$ are Rodriguez-Villegas transforms of $\sum_{j=0}^{k-3} x^j$ and that the $H^+_k(x)$ are the transforms of $x^{k-2}+1$. 

The fact that the zeros of $H^{-}_{k}(x)$ lie on the line $\re(x)=1/2$ is one of the main examples of Theorem \ref{RVThm} in \cite{BHW}. The precise location of its zeros in Theorem \ref{ZeroDistribution} is a direct restatement (and slight modification in the case of $\epsilon(f)=+1$) of Theorem 1.7 of \cite{BHW}. 
 
The proof of part (2) follows directly from the more precise estimates on the locations (and in particular the equidistribution) of the zeros in Theorem \ref{RHPP}, which is hinted at in Remark \ref{PreciseRmk} and provided in Theorem 1.2 in \cite{RHPP}. As extracting roots of a polynomial is continuous in the coefficients of the polynomial, we have the desired convergence of the roots of $Z_f(s)$ in the limit to those of the polynomials $H^{\pm}_k(x)$.

Part (3) also follows by utilizing the precise estimates of Theorem 1.2 of \cite{RHPP} to determine the zeros of $R_f(z)$ to high accuracy. 
The strict upper bound on the imaginary parts of roots follows from a careful consideration of  Theorem 1 of \cite{B}, using the positivity properties of critical completed $L$-values reviewed in Lemma \ref{ValueOneNotZeroPP}. When $\epsilon(f)=-1$, by Lemma \ref{ValueOneNotZeroPP} and the functional equation for $\Lambda(f,s)$, we see that $R_f(x)/(1-x)$ is a polynomial with all non-positive coefficients, so the Rodriguez-Villegas transform of $R_f$ is the same as the Rodriguez-Villegas transform of this polynomial with the degree lowered by $1$. This shows that we may again use Theorem 1 of \cite{B}, by applying it to the polynomial $R_f(s)/(x-1)$ with positive coefficients and whose transform has the same zeros as $Z_f(s)$.

\end{proof}

\subsection{Examples}\label{RHPPExamples}

\subsubsection{Period polynomials in weight 4}
To give a flavor of the types of analysis going into the proof of Theorem \ref{RHPP} for small weights and levels, here we consider the case of weight $4$ newforms. 
The period polynomial $R_f$ is then quadratic, and we are concerned with the zeros of
$$
-2L(f,1)\pi^2 X^2+2\pi i L(f,2)X+L(f,3)=0.
$$

Clearly, we may then apply the quadratic formula.
Now, the values $L(f,1)$ and $L(f,3)$ are also related via the functional equation for $L(f,s)$, and the 
conclusion is thus trivial if $L(f,2)=0$.

In the case when $L(f,2)\neq 0$, the theorem is equivalent to the estimate
\[\frac{N}{\pi^2} L(f,3)^2 \geq L(f,2)^2.\]
This can be shown using the Hadamard factorization of $\Lambda(f,s)$:
$$
\Lambda(f,s)=e^{A+Bs}\prod_{\rho}\left(1-\frac{s}{\rho}\right)\exp(s/\rho)
$$
Here the product is over all the zeros of $\Lambda(f,s)$ (that is, the non-trivial zeros of $L(f,s)$), and $A$ and $B$ are constants.  
Now we always have $3/2\leq \re(\rho)\leq 5/2$, which means that $\Lambda(f,3)\geq \Lambda(f,2)$, and is enough to imply the desired inequalities.

\subsubsection{Zeta function for the modular discriminant}

It is interesting to consider the first level 1 newform, namely, the normalized Hecke eigenform $f=\Delta\in S_{12}(\Gamma_0(1))$. In this case, $\epsilon(f)=1$ and we have
\begin{align*}
R_{\Delta}(z)
\approx\ 
&
0.114379\cdot\left(\frac{36}{691}z^{10}+z^8+3z^6+3z^4+z^2+\frac{36}{691}\right)
\\
&
\ \ \ \ \ \ +0.00926927\cdot(4z^9+25z^7+42z^5+25z^3+4z).
\end{align*}
The ten zeros of $R_{\Delta}$ lie on the unit circle, and are approximated by the set
$$
\left\{
\pm i,-0.465\pm 0.885i,\
 -0.744\pm 0.668i,\
-0.911 \pm0.411i,\
-0.990\pm0.140i
\right\}.
$$
These are illustrated in the following diagram.
\begin{figure}[H]
\centering
\includegraphics[scale=0.5,keepaspectratio=true]{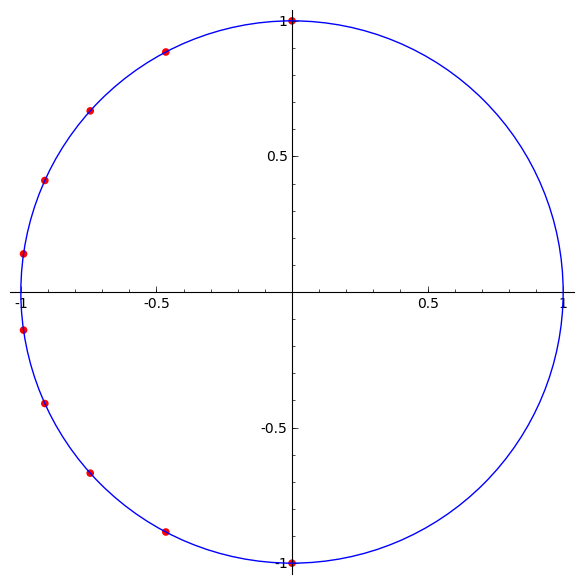}
\caption{The roots of $R_{\Delta}(z)$}
\end{figure}

It is quite tempting from this picture to conjecture that the zeros of $R_f$ always lie on one half of a unit circle. However, this small example is somewhat misleading: the zeros for large level or weight modular forms eventually have to equidistribute on the whole unit circle, as shown in \cite{RHPP}. Another manner in which this example is atypical is the following. The size of the coefficients of $R_f$ follows a unimodal pattern. The same equidistribution-of-zeros result shows that eventually the coefficients of $R_f$ must be unimodal, but in the opposite direction! In fact, if the direction of this unimodality of coefficients were reversed, the Riemann Hypothesis for $R_f$ would follow directly from its functional equation (showing it is a self-inversive polynomial) along with elementary results such as Theorem 1 of \cite{Chen}. It is this eventual ``correct'' unimodality that follows from the results of \cite{RHPP}, and this drastic switching of phenomena that make the proof of Theorem \ref{RHPP} split into several cases and make its statement surprising. In fact, this line of reasoning allows one to conclude that for large weights and levels, the location of the zeros of $R_f(s)$ on the unit circle is a relatively ``stable'' phenomenon, while a quick numerical experiment for a polynomial such as $R_{\Delta}$ shows that the location of its zeros on the circle is very unstable, and hence rather astounding.

We now consider the Rodriguez-Villegas transform; letting $s\mapsto-s$ we find that 
\begin{align*}
&
Z_{\Delta}(s)\approx
(5.11\times 10^{-7})s^{10}
-(2.554\times 10^{-6})s^9 +(6.01\times 10^{-5})s^8
-( 2.25\times 10^{-4})s^7 
\\
&
+0.00180s^6 
-0.00463s^5+0.0155s^4
-0.0235s^3 +0.0310s^2
-0.0199 s
+
0.00596
  .
\end{align*} 
Theorem~\ref{Thm1} establishes that the zeros $\rho$ of $Z_{\Delta}$ satisfy $\re(\rho)=1/2$; indeed, they are  approximately
\begin{align*}
\left\{
\frac12\pm 8.447i,
 \frac12\pm 5.002i,
\frac12 \pm 2.846i,
\frac12 \pm 1.352i,
\frac12 \pm 0.349i,
\right\}
,
\end{align*}
as illustrated in the next figure.
\begin{figure}[H]
\centering
\includegraphics[scale=0.5,keepaspectratio=true]{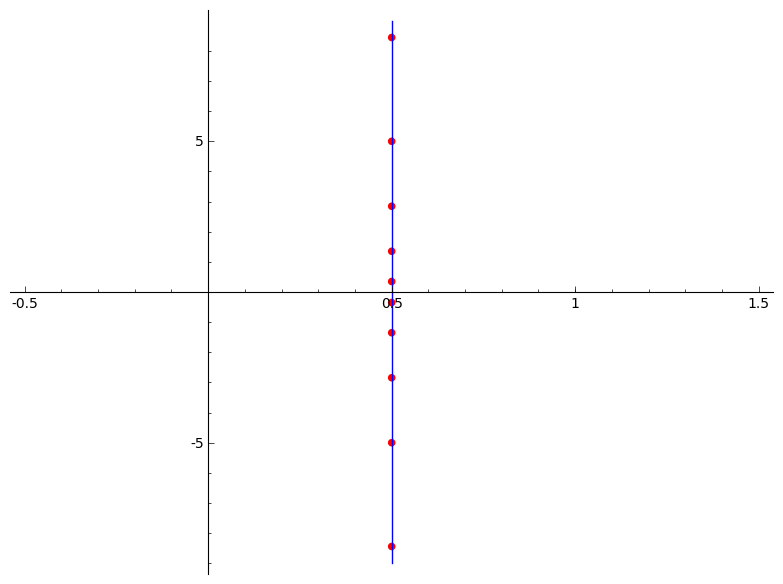}
\caption{The roots of $Z_{\Delta}(s)$}
\end{figure}

\subsubsection{Ehrhart polynomials and newforms of weight $6$}

Here we consider newforms $f\in S_6(\Gamma_0(N))$ with $\epsilon(f)=-1$.  
By Theorem~\ref{ZeroDistribution} (2), 
the roots of $Z_f(s)$ may be modeled (with improving accuracy as $N\rightarrow\infty$) on the roots of the Ehrhart polynomial of the convex hull 
\[
\operatorname{conv}\left\{e_1,e_2,e_3,-e_1-e_2-e_3\right\}
.
\]  
The following image image displays this tetrahedron. 
\begin{figure}[H]\label{FigureTetra}
\centering
\includegraphics[scale=0.5,keepaspectratio=true]{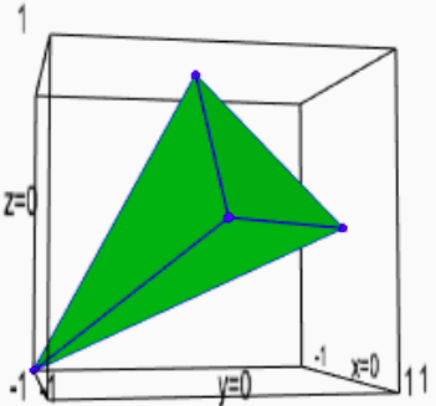}
\caption{The tetrahedron whose Ehrhart polynomial is $H_6^-(s)$.}
\end{figure}

The corresponding Ehrhart polynomial counts the number of integer points in dilations of Figure \ref{FigureTetra}, and is given by the Rodriguez-Villegas transform of 
$1+x+x^2+x^3$. Namely, we have
\[
H_6^-(s)=
\binom{s+3}{3}+\binom{s+2}{3}+\binom{s+1}{3}+\binom{s}{3}
=\frac23s^3 + s^2 + \frac73s + 1
.
\]
Therefore, we find that
$$
\lim_{N\rightarrow +\infty} \widetilde{Z}_f(s)=\widetilde{H}^{-}_6(-s)= \left(s-\frac{1}{2}\right)
\left(s-\frac{1}{2}+\frac{\sqrt{-11}}{2}\right) \left(s-\frac{1}{2}-\frac{\sqrt{-11}}{2}\right),
$$
where the limit is over newforms $f\in S_6(\Gamma_0(N))$ with $\epsilon(f)=-1$, and where the polynomials with tildes have been normalized to have leading coefficient 1.

\subsection{Concluding remark}

In both this and the previous section, we have surveyed a number of natural partition-theoretic forms emerging from the same stellar nursery as the Riemann zeta function. Indeed, it is evident that classical zeta functions represent a relatively small subclass in the vast universe of partition-theoretic objects.
%{\bf LR: Robert, Are you going to add a bit more here? It seems a little short as an epilogue.}

\  
\  
\  
\  
\  
\  

%{\bf LR: Robert, reference KM says to be added; do you know which one this is supposed to be? RPS: This was one you had mentioned from when we first drafted the paper.}

%\input{referenc}


\vspace{10mm}
\begin{thebibliography}{BrStr}
\bibitem{AlladiErdos} K. Alladi and P. Erd\H{o}s, {\it On an additive arithmetic function}, Pacific J. Math. {\bf 71} (1977), no. 2, 275-294.
\bibitem{Andrews} G. Andrews, {\it The theory of partitions}, Reprint of the 1976 original, Cambridge Mathematical Library, Cambridge University Press, Cambridge, 1998.
\bibitem{Berger}L. Berger, \emph{An introduction to the theory of p-adic representations}, Geometric aspects of Dwork theory. Vol. I, 255--292, Walter de Gruyter, 2004.
\bibitem{Besser} A. Besser and H. Furusho, {\it The double shuffle relations for p-adic multiple zeta values}, Primes and knots, 9?29, Contemp. Math {\bf 416}.
\bibitem{BHW} C. Bey, M. Henk, and J. Wills, \emph{Notes on the roots of Ehrhart polynomials}, Discrete and Comput. Geom. {\bf 38} (2007), 81--98.
\bibitem{BK} S. Bloch and K. Kato, \emph{$L$-functions and Tamagawa numbers of motives}, Grothendieck Festschrift, Vol 1 (1990), Birkh\"auser, 333--400.
\bibitem{BlochOkounkov} S. Bloch and A. Okounkov, {\it The character of the infinite wedge representation}, Adv. Math. {\bf 149} (2000), no. 1, 1-60.
\bibitem{BorweinBradley} J. Borwein and D. Bradley, {\it Evaluations of $k$-fold Euler/Zagier sums: a compendium of results for arbitrary $k$}. The Wilf Festschrift (Philadelphia, PA, 1996), Electron. J. Combin. {\bf 4} (1997), no. 2, Research Paper 5.
\bibitem{BorweinZudilin} J. Borwein and W. Zudilin, {\it Multiple zeta values}, \url{https://carma.newcastle.edu.au/MZVs/} (2011).
\bibitem{B} B. Braun, \emph{Norm bounds for Ehrhart polynomial roots}, Discrete and Comput. Geom. {\bf 39} (2008), no. 1--3, 191--191.
\bibitem{BKO} J. H. Bruinier, W. Kohnen, and K. Ono. {\it The arithmetic of the values of modular functions and the divisors of modular forms}. Composition Math. \textbf{140} (2004), 552-566.
\bibitem{CFI} J.B. Conrey, D.W. Farmer, O. Imamo\u{g}lu, \emph{The nontrivial zeros of period polynomials of modular forms lie on the unit circle}, Int. Math. Res. Not. IMRN  2013,  no. 20, 4758--4771.
\bibitem{Chen} W.Y. Chen, {\it On the Polynomials with All Their Zeros on the Unit Circle,} J. of Math. Anal. and Appl. {\bf 190} (1995), 714--724.
\bibitem{ChamberlandJohnsonNadeauWu} M. Chamberland, C. Johnson, A. Nadeau, and B. Wu, {\it Multiplicative partitions}, Electron. J. Combin. {\bf 20} (2013), no. 2, Paper 57.  
\bibitem{ChamberlandStraub} M. Chamberland and A. Straub, {\it On gamma quotients and infinite products}, Adv. in Appl. Math. {\bf 51} (2013), no. 5, 546--562. 
\bibitem{CPZ} Y. Choie, Y. K. Park, and D. Zagier, \emph{Periods of modular forms on $\Gamma_0(N)$ and products of Jacobi theta functions},
preprint.
\bibitem{DingFengLiu} S. Ding, L. Feng, and W. Liu, {\it A combinatorial identity of multiple zeta values with even arguments}, Electron. J. Combin. {\bf 21} (2014), no. 2, Paper 2.27.
\bibitem{D} N. Dummigan, \emph{Congruences of modular forms and Selmer groups}, Math. Res. Letters, no. 8, 479--494 (2001).
\bibitem{Dunham} W. Dunham, {\it Euler: the master of us all}, The Dolciani Mathematical Expositions, 22. Mathematical Association of America, Washington, DC, 1999.
\bibitem{Edwards} H. M. Edwards, {\it Riemann's zeta function}, Reprint of the 1974 original, Dover Publications, Inc., Mineola, NY, 2001.
\bibitem{ER} A. El-Guindy, W. Raji, \emph{Unimodularity of zeros of period polynomials of Hecke eigenforms}, Bull. Lond. Math. Soc.  46  (2014),  no. 3, 528--536.
\bibitem{diBruno} F. Fa\`{a} di Bruno, {\it Sullo sviluppo delle funzioni}, Annali di Scienze Matematiche e Fisiche {\bf 6} (1855): 479-480.
\bibitem{Fine} N. J. Fine, {\it Basic hypergeometric series and applications},With a foreword by George E. Andrews, Mathematical Surveys and Monographs, 27. American Mathematical Society, Providence, RI, 1988.
\bibitem{FG} E. M. Friedlander and D. R. Grayson, \emph{Handbook of $K$-Theory}, Springer, New York, 2000.
\bibitem{Furusho} H. Furusho, {\it $p$-adic multiple zeta values I -- $p$-adic multiple polylogarithms and the $p$-adic KZ equation}, Invent. Math. {\bf 155} (2004), no. 2, 253--286.
\bibitem{GRV} P. Gunnells and F. Rodriguez-Villegas, \emph{Lattice polytopes, Hecke operators, and the Ehrhart polynomial} Sel. math., New ser. {\bf 13} (2007), 253--276.
\bibitem{Hoffman1} M. Hoffman, {\it Multiple harmonic series}, Pacific J. Math. {\bf 152} (1992), no. 2, 275-290.
\bibitem{Hoffman2} M. Hoffman, {\it Multiple zeta values}, \url{http://www.usna.edu/Users/math/meh/mult.html} (2012).
\bibitem{IKZ} K. Ihara, M. Kaneko, and D. Zagier, {\it Derivation and double shuffle relations for multiple zeta values}, Compos. Math. {\bf 142} (2006), 307--338.
\bibitem{RHPP} S. Jin, W. Ma, K. Ono, and K. Soundararajan, {\it The Riemann Hypothesis for period polynomials of modular forms,} Proceedings of the National Academy of Sciences, USA, {\bf 113} No. 10 (2016), 2603--2608.
\bibitem{KM} A. Knopfmacher and M. E. Mays, {\it Compositions with $m$ distinct parts}, Ars Combinatoria {\bf 53} (1999): 111-128.
\bibitem{KL} T. Kubota and H.-W. Leopoldt, {\it Eine p-adische Theorie der Zetawerte. I. Einführung der $p$-adischen Dirichletschen $L$-Funktionen}, Journal f\"ur die reine und angewandte Mathematik {\bf 214/215} (1964), 328--339.
\bibitem{KZ} W. Kohnen and D. Zagier, \emph{Modular forms with rational periods}, Modular forms
(Durham, 1983), Ellis Horwood Ser. Math. Appl.: Statst. Oper. Res., Horwood, Chichester, 1984, 197-249.
\bibitem{MacMahon} P. MacMahon, {\it Dirichlet series and the theory of partitions}, Proc. London Math. Soc. S2-22 (1924) no. 1, 404-411.  
\bibitem{M} Y. I. Manin, \emph{Local zeta factors and geometries under ${\text {\rm Spec}} {\bf \ Z}$},
Izv. Russian Acad. Sci. (Volume dedicated to J.-P. Serre), accepted for publication, arXiv:1407.4969.
\bibitem{Nist} NIST Digital Library of Mathematical Functions, \url{http://dlmf.nist.gov/}, Release 1.0.6 of 2013-05-06.
\bibitem{ORS} K. Ono, L. Rolen, and F. E. Sprung, {\it Zeta-polynomials for modular form periods}, submitted for publication.
\bibitem{PP} V. Pa\c sol, A. Popa, \emph{Modular forms and period polynomials}, Proc. Lond. Math. Soc. (3)  107  (2013),  no. 4, 713--743.
\bibitem{Polya} G. P{\' o}lya, \emph{{\" U}ber die Nullstellen gewisser ganzer Funktionen} [On the zeros of certain entire functions], Math. Zeit., vol. 2 (1918),  353--383.  German
\bibitem{RV} F. Rodriguez-Villegas, \emph{On the zeros of certain polynomials}, Proc. Amer. Math. Soc. \textbf{130}, No. 5 (2002), 2251-2254.
\bibitem{Robert} R. P. Schneider, {\it Partition zeta functions}, Preprint (2015). 
\bibitem{scholl} T. Scholl, \emph{Motives for modular forms}, Inventiones \textbf{100}, (1990), 419--430.
\bibitem{Subbarao} M. Subbarao, {\it Product partitions and recursion formulae}, Int. J. Math. Math. Sci. 2004, no. 33-36, 1725--1735. 
\bibitem{Szego} G. Szeg{\" o}, \emph{Inequalities for the zeros of Legendre polynomials and related functions}, Trans. Amer. Math. Soc. vol. 39 (1), (1936),  1--17.
\bibitem{Tsumura} H. Tsumura, {\it Combinatorial relations for Euler-Zagier sums,} Acta Arith. {\bf 111} (2004), 27--42.
\bibitem{W} J.-L. Waldspurger, {\it Sur les valeurs de certaines fonctions L automorphes en leur centre de sym\'etrie}, Compositio
Math. 54 (1985), 173--242.
\bibitem{WhitttakerWatson}
  E.  T.  Whittaker and G.  N.  Watson,  {\it A  Course  of  Modern  Analysis},  4th  ed.,
Cambridge University Press, Cambridge, 1927.
\bibitem{Wrench} J. Wrench, {\it Concerning two series for the gamma function}, Math. Comp. {\bf 22} (1968): 617-626. 
\bibitem{Zagier} D. Zagier, {\it Evaluation of the multiple zeta values $\zeta(2,...,2,3,2,...,2)$}, Ann. of Math. (2) {\bf 175} (2012), no. 2, 977--1000. 
\bibitem{ZagHecke}  D. Zagier, \emph{Hecke operators and periods of modular forms}, in
Festschrift in honor of I. I. Piatetski-Shapiro on the occasion of his sixtieth birthday, Part II (Ramat Aviv, 1989), volume 3 of Israel Math.
Conf. Proc., pages 321--336. Weizmann, Jerusalem, 1990.
\bibitem{ZagierMZV} D. Zagier, {\it Multiple zeta values}, Manuscript in preparation (1995).
\bibitem{Zagier2} D. Zagier, {\it Partitions, quasi-modular forms, and the Bloch-Okounkov theorem}, Preprint (2015).
\bibitem{Z} D. Zagier, \emph{Periods of modular forms and Jacobi theta functions}, Invent. Math. \textbf{104} (1991), 449-465.
\bibitem{ZaharescuZaki} A. Zaharescu and M. Zak, {\it On the parity of the number of multiplicative partitions}, Acta Arith. {\bf 145} (2010), no. 3, 221-232. 


\end{thebibliography}
\end{document}